\documentclass[a4paper,reqno]{amsart}

\usepackage{amsmath,amsfonts,amsthm,amssymb,euler,newpxtext,latexsym,mathrsfs,enumerate,tikz-cd}

\usepackage[colorlinks=true,linkcolor=blue,citecolor=red,urlcolor=teal,bookmarksdepth=subsection,final]{hyperref}

\DeclareMathOperator{\conv}{\mathrm{conv}}
\DeclareMathOperator{\tr}{\mathrm{Tr}}

\DeclareMathOperator{\ev}{\mathrm{ev}}
\DeclareMathOperator{\id}{\mathrm{id}}
\DeclareMathOperator{\coker}{\mathrm{coker}}
\DeclareMathOperator{\Aut}{\mathrm{Aut}}
\DeclareMathOperator{\End}{\mathrm{End}}
\DeclareMathOperator{\Ell}{\mathrm{Ell}}
\DeclareMathOperator{\aff}{\mathrm{Aff}}

\DeclareMathOperator{\prob}{\mathcal{M}^1_+}
\DeclareMathOperator{\lip}{\mathrm{Lip}^1}
\DeclareMathOperator{\emb}{\mathrm{Emb}^1}
\DeclareMathOperator{\Span}{\mathrm{span}}
\DeclareMathOperator{\dist}{\mathrm{dist}}
\DeclareMathOperator{\dq}{\mathrm{dist}_{\mathrm{q}}}
\DeclareMathOperator{\isom}{\mathrm{Isom}_{\varepsilon}}
\DeclareMathOperator*{\free}{{\scalebox{1.7}{$\ast$}}}

\newcommand{\nn}{\mathbb{N}}
\newcommand{\zz}{\mathbb{Z}}
\newcommand{\qq}{\mathbb{Q}}
\newcommand{\rr}{\mathbb{R}}
\newcommand{\cc}{\mathbb{C}}
\newcommand{\pp}{\mathbb{P}}
\newcommand{\ee}{\mathbb{E}}
\newcommand{\ff}{\mathbb{F}}

\newcommand{\cs}{\mathrm{C}^*}
\newcommand{\js}{\mathcal{Z}}

\makeatletter
\@namedef{subjclassname@2020}{%
  \textup{2020} Mathematics Subject Classification}
\makeatother

\newtheorem {theorem}{Theorem}[section]

\newtheorem {proposition}[theorem]{Proposition}
\newtheorem {corollary}[theorem]{Corollary}

\newtheorem {thm}{Theorem}

\theoremstyle {definition}

\newtheorem {example}[theorem]{Example}
\newtheorem {definition}[theorem]{Definition}
\newtheorem {remark}[theorem]{Remark}
\newtheorem {question}[theorem]{Question}
\newtheorem {notation}[theorem]{Notation}

\numberwithin{equation}{section}

\theoremstyle{nonumberplain}
\newtheorem* {claim}{Theorem}

\title{Quantum metric Choquet simplices}

\author[B.~Jacelon]{Bhishan Jacelon}
\address[B.~Jacelon]{
Institute of Mathematics of the Czech Academy of Sciences\\ \v{Z}itn\'{a} 25\\115 67 Prague 1\\Czech Republic}
\email{\href{mailto:jacelon@math.cas.cz}{jacelon@math.cas.cz}}


\begin{document}

\begin{abstract} 
Precipitating a notion emerging from recent research, we formalise the study of a special class of compact quantum metric spaces. Abstractly, the additional requirement we impose on the underlying order unit spaces is the Riesz interpolation property. In practice, this means that a `quantum metric Choquet simplex' arises as a unital $\cs$-algebra $A$ whose trace space is equipped with a metric inducing the $w^*$-topology, such that tracially Lipschitz elements are dense in $A$. This added structure is designed for measuring distances in and around the category of stably finite classifiable $\cs$-algebras, and in particular for witnessing metric and statistical properties of the space of approximate unitary equivalence classes of unital embeddings of $A$ into a stably finite classifiable $\cs$-algebra $B$. As for examples, we recall the construction of classifiable $\cs$-algebraic quantum metric Bauer simplices that function as noncommutative spaces of observables of compact connected metric spaces $(X,\rho)$. We also explain how to build non-Bauer examples by forming `tracial quantum crossed products' associated with topological dynamical systems on $(X,\rho)$, and we use classification to show that continuous fields of quantum spaces are obtained by continuously varying either the dynamics or the metric. In the case of deformed isometric actions, we show that equivariant Gromov--Hausdorff continuity implies fibrewise continuity of the quantum structures with respect to Rieffel's quantum Gromov--Hausdorff distance. As an example, we present a field of deformed tracial rotation algebras whose fibres are continuous with respect to a quasimetric that we call the quantum intertwining gap.
\end{abstract}

\maketitle

\section{Introduction} \label{section:intro}

This paper's purpose is an abstraction of the metric structure found in certain model $\cs$-algebras that has been used in recent analysis \cite{Jacelon:2021wa,Jacelon:2021vc,Jacelon:2022wr} of morphisms into the category of classifiable $\cs$-algebras. Here, `classifiable' means via the Elliott invariant, and we frequently use this term as an abbreviation for `simple, separable, nuclear, $\js$-stable and satisfying the universal coefficient theorem (UCT)' (see \cite{Gong:2020ud,Gong:2020uf,Elliott:2020wc,Gong:2020vg,Gong:2021tw,Elliott:2016ab} and also \cite{Carrion:wz}). The Elliott invariant $\Ell$ consists of $K$-theory, traces and the pairing between them, and we will make liberal use of it and its components throughout the article. This should not deter nonexperts, however, as our practical use of classification theory does not require familiarity with the intricate inner workings of its anatomy.

The metric structure in question is extraneous to what is guaranteed by $\cs$-algebraic axioms, and the framework for our analysis is Rieffel's notion of a \emph{compact quantum metric space} \cite{Rieffel:2004aa}. Abstractly, we additionally require the Riesz interpolation property (as well as completeness). In the $\cs$-algebraic setting, this means that our function space is not the self-adjoint part $A_{sa}$ of a $\cs$-algebra $A$ but rather its Cuntz--Pedersen quotient $A^q$, and in practice this means that we are interested in metrics inducing the $w^*$-topology on the \emph{trace} space $T(A)$, rather than the \emph{state} space $S(A)$ (see Theorem~\ref{thm:cqmcs}). A suitable \emph{nucleus} $\mathcal{D}_r(A)$ of tracially Lipschitz elements (see Theorem~\ref{thm:nucleus}) then provides a basis for distance measurement (as in Section~\ref{subsubsection:unitary}, Example~\ref{ex:metricdeform2} and \cite{Jacelon:2021wa,Jacelon:2021vc}) or the observation of statistical properties of tracial dynamics (as in Section~\ref{subsubsection:birkhoff}, Section~\ref{subsection:random} and \cite{Jacelon:2022wr}).

Our base of examples comes from \cite[Theorem 4.4 and Remark 4.5]{Jacelon:2022wr}, restated here for convenience. Quantum metric \emph{Bauer} simplices are special quantum metric Choquet simplices for which the metric on the trace space has been induced in a canonical way from a metric on the compact boundary (see Definition~\ref{def:qmbs}).

\begin{claim}[\cite{Jacelon:2022wr}]
For every compact metric space $(X,\rho)$, there is a $\cs$-algebraic quantum metric Bauer simplex $(\mathcal{A}_X,L_\rho)$ such that $\mathcal{A}_X$ is unital and classifiable with extreme tracial boundary $\partial_e(T(\mathcal{A}_X))$ homeomorphic to $X$, and such that the seminorm $L_\rho$ induces the metric $\rho$. If $X$ is connected, then $\mathcal{A}_X$ can be chosen to be $K$-connected, and if moreover $X$ is a Riemannian manifold with $\rho$ the associated intrinsic metric, then $\mathcal{A}_X$ can be chosen to be projectionless.
\end{claim}

The projectionless models associated with Riemannian manifolds are denoted by $\js_X$, because we consider them to be tracially higher-dimensional analogues of the Jiang--Su algebra $\js$ within the category of classifiable quantum metric Bauer simplices (built as they are out of generalised dimension drop algebras over the space $X$). These structures are of particular interest because, to the extent that projections represent obstacles, there is maximal opportunity for a rich space of embeddings into a given classifiable $\cs$-algebra $B$.

Our interpretation of the above theorem, and the overarching viewpoint of the article, is that the $\cs$-algebra $\mathcal{A}_X$ serves as a noncommutative space of observables of the metric space $(X,\rho)$. In Rieffel's work, `quantum' is used to describe exactly this sort of situation in which order unit spaces (for us, $\mathcal{A}_X^q$) observe their state spaces (for us, $T(\mathcal{A}_X)$). Rieffel opts for `quantum' over `noncommutative' because the multiplicative structure of the ambient $\cs$-algebra \emph{a priori} plays no role in statial observation. But through classification \cite{Carrion:wz}, the $\cs$-algebraic structure \emph{does} come into play: topological dynamical systems $(X,h)$ can be lifted to $\cs$-dynamical systems $(\mathcal{A}_X,\alpha_h)$ whenevver $\mathcal{A}_X$ is classifiable and \emph{$K$-connected}, which is our replacement for what is referred to in \cite[Definition  3.3]{Jacelon:2022wr} as having \emph{trivial tracial pairing}. Its definition is that the ordered $K_0$-group of the minimal unitisation should admit a unique state. We think of this as a connectedness condition since an abelian $\cs$-algebra $C_0(X)$ is $K$-connected precisely when $X\cup\{\infty\}$ is connected. This property allows for the tracial and $K$-theoretic parts of the classifying invariant to be manipulated independently: any continuous affine map $T(A) \to T(A)$ is automatically compatible with any morphism $K_*(A) \to K_*(A)$, so together they provide a morphism $\Ell(A) \to \Ell(A)$ that can then be lifted via classification \cite[Corollary C]{Carrion:wz} to an endomorphism of $A$.

Throughout the article, our use of the word `quantum' is in keeping with Rieffel's (rather than having anything to do with quantum physics or quantum groups). To emphasise that our observations are specifically made via the evaluation map $A\to\aff(T(A))$, we may also include the word `tracial'. So, a \emph{tracial quantum system associated with a topological dynamical system $(X,h)$} means a $\cs$-dynamical system $(A_X,\alpha_h)$, where $A_X$ is a $\cs$-algebra with $\partial_e(T(A_X))\cong X$ (not necessarily one of the classifiable models $\mathcal{A}_X$) and $\alpha_h\colon A_X \to A_X$ is a $^*$-homomorphism such that $T(\alpha_h)|_{\partial_e(T(A_X))}=h$. By a \emph{tracial quantum crossed product associated with $(X,h)$} we mean that the crossed product $C(X)\rtimes_{h^*}\nn$ that one classically attaches to $(X,h)$ has been replaced by $A_X\rtimes_{\alpha_h}\nn$ (see Definition~\ref{def:qcp}, as well as Remark~\ref{rem:groups}, Theorem~\ref{thm:free} and Example~\ref{ex:sphere}, which consider group actions $G\to\Aut(C(X))$ beyond $G=\zz$).

In $\cs$-algebraic analyses of topological dynamical systems $(X,h)$, properties like minimality and freeness of $h$ are used to regulate the structure of the crossed product $C(X)\rtimes_{h^*}\nn$. In the present work, as in \cite{Jacelon:2022wr}, we replace such topological hypotheses by \emph{finite Rokhlin dimension}, which is a regularity property of the noncommutative system $\alpha_h\colon A_X \to A_X$ that is in fact generic whenever $A_X$ is classifiable (see \cite{Hirshberg:2015wh,Szabo:2019te} and also \cite[Definition 4.6, Lemma 4.7]{Jacelon:2022wr}). With this extra ingredient in place, much can be said about the structure of $A_X\rtimes_{\alpha_h}\nn$. In particular, $A_X\rtimes_{\alpha_h}\nn$ is itself classifiable with $T(A_X\rtimes_{\alpha_h}\nn) \cong \prob(X)^h$, the simplex of $h$-invariant Borel probability measures on $X$, and if $A_X$ is equipped with $\cs$-algebraic quantum metric Bauer structure associated with the metric space $(X,\rho)$, then $A\rtimes_{\alpha_h}\nn$ becomes in a natural way a quantum metric Choquet simplex in its own right (see Proposition~\ref{prop:inherit}). 

Inspired by \cite{Kaad:2021aa}, one of our main tasks is to analyse the variation of this inherited structure as we continuously deform either the metric $\rho$ on $X$ (Theorem~\ref{thm:metricdeform}) or the dynamics $(X,h)$ (Theorem~\ref{thm:cf}). A consequence of  \cite{Carrion:2024aa} is that the varying structures $(A_X\rtimes_{\alpha_{h_\theta}}\nn)_{\theta\in\Theta}$ can be assembled into continuous fields of quantum spaces (see Definition~\ref{def:cf}).

\begin{thm} \label{thm2}
Let $(X,\rho)$ be a compact, connected metric space, $\Theta$ a compact metrisable space and $(h_\theta)_{\theta\in\Theta}$ a family of uniformly Lipschitz homeomorphisms $X\to X$ that varies continuously with respect to the equivariant Gromov--Hausdorff topology. Then, there exists a unital classifiable $\cs$-algebraic quantum metric Bauer simplex $A_X$ associated with $(X,\rho)$, and a point-norm continuous family of tracial quantum systems $(A_X,\alpha_\theta)_{\theta\in\Theta}$ associated with $(X,h_\theta)_{\theta\in\Theta}$, such that $((A_X\rtimes_{\alpha_\theta}\nn)_{\theta\in\Theta},C(\Theta,A_X)\rtimes_{\{\alpha_\theta\}}\nn)$ forms a continuous field of unital classifiable $\cs$-algebraic quantum metric Choquet simplices whose fibres vary continuously with respect to the quantum Gromov--Hausdorff distance.
\end{thm}

Example~\ref{ex:rotation} describes the field of deformed tracial rotation algebras, a special instance of Theorem~\ref{thm2} in which the \emph{quantum intertwining gap} $\gamma_q$ between the fibres is also continuous. This variant of Rieffel's quantum Gromov--Hausdorff distance $\dq$ (see Definition~\ref{def:qgh}) is introduced in Definition~\ref{def:quig} alongside our analysis of the metric space $\mathbb{QMCS}$ of isometry classes of quantum metric Choquet simplices and the subspace $\mathbb{QMBS}$ of isometry classes of quantum metric Bauer simplices. (See also Corollary~\ref{cor:quig}, which in the classifiable Bauer setting provides a `quantum' interpretation of $\gamma_q$, that is, a description of $\gamma_q$ at the $\cs$-algebraic level.) Following Rieffel, we show in Theorem~\ref{thm:quig} and Corollary~\ref{cor:complete} that $(\mathbb{QMBS},\dq)$ is complete. The corresponding question for $\mathbb{QMCS}$ is left open (see Question~\ref{q:qmcs}).

A well-known deficiency of the distance $\dq$ is that it does not capture $\cs$-algebraic data. In particular, while distance zero between $\cs$-algebraic compact quantum metric spaces $(A,L)$ and $(B,M)$ necessitates the existence of an order-theoretic isomorphism between the self-adjoint parts of $A$ and $B$, this in general does not imply that $A$ and $B$ are isomorphic as $\cs$-algebras. In contrast to earlier works like \cite{Kerr:2009aa} and \cite{Latremoliere:2016wn}, our approach to addressing this discrepancy is to use classification to solve a $K$-theoretically `localised' version of the problem (see Corollary~\ref{cor:quig}).

\begin{thm} \label{thm1}
Let $\mathcal{K}$ denote the set of isometric isomorphism classes of $\cs$-algebraic quantum metric Bauer simplices $(A,L)$, where $A$ is a unital, $K$-connected, classifiable $\cs$-algebra and $(A,L)$ is equivalent to $(B,M)$ if there is a $^*$-isomorphism $\varphi \colon A \to B$ such that $M \circ \varphi = L$. Given a countable abelian group $G_1$ and a countable, simple, weakly unperforated ordered abelian group $(G_0,G_0^+)$ with distinguished order unit $g_0\in G_0^+$, write $G= (G_0,G_0^+,g_0,G_1)$ and
\[
\mathcal{K}_{G} = \{[(A,L)] \in \mathcal{K} \mid (K_0(A),K_0(A)_+,[1_A],K_1(A)) \cong G\}.
\]
Then, for every such $G$, $\gamma_q$ is a quasimetric on $\mathcal{K}_{G}$ that is topologically equivalent  to $\dq$.
\end{thm}

We do not always insist on unitality and indeed in Section~\ref{section:qcp} we sometimes rely on model $\cs$-algebras that are stably projectionless, in particular the $\mathcal{W}$-stable algebras classified in \cite{Elliott:2020wc} and the $\js_0$-stable algebras classified in \cite{Gong:2020vg}. Here, $\mathcal{W}$ and $\js_0$ are stably projectionless analogues of the Cuntz algebras $\mathcal{O}_2$ and $\mathcal{O}_\infty$ respectively (and, like $\js$, can also be described as Fra\"{i}ss\'{e} limits---see \cite{Jacelon:2021uc}). Tensorial absorption of $\mathcal{W}$ entails trivial $K$-theory (which under the UCT is equivalent to `$KK$-contractibility') and tensorial absorption of $\js_0$ is a weaker condition that necessitates $K$-connectedness (which for nonunital $\cs$-algebras means that $K_0(A) \subseteq \ker\rho_A$, where $\rho_A\colon K_0(A) \to \aff(T(A))$ denotes the pairing map between $K$-theory and traces). In this stably projectionless setting, it is possible to compute the Elliott invariant of $A_X\rtimes_{\alpha_h}\nn$ and therefore determine its (stable) isomorphism class. Indeed, the proof of \cite[Theorem B]{Jacelon:2022wr} demonstrates that every $\mathcal{W}$-stable classifiable $\cs$-algebra is stably isomorphic to a tracial quantum crossed product associated with a minimal subshift (in fact, an almost 1-1 extension of the dyadic odometer). Essentially the same argument yields a more general result (see Theorem~\ref{thm:range}).

\begin{thm} \label{thm3}
Every $\js_0$-stable classifiable $\cs$-algebra is stably isomorphic to a tracial quantum crossed product associated with a minimal homeomorphism of a zero-dimensional space.
\end{thm}

One final comment we should make is on the difference between our \emph{quantum} Choquet simplices and the \emph{noncommutative} Choquet simplices of \cite{Kennedy:2022aa}. This closeness of terminology genuinely was unintentional, but there is an illuminating explanation for our progression along a parallel path. The starting point in both cases is the function system $\aff(K)$ of continuous affine functions on a compact, convex set $K$. The framework of \cite{Kennedy:2022aa} is the noncommutative convexity theory developed in \cite{Davidson:2022aa} using the language of operator systems. The other available avenue in the separable setting, the one down which we have followed Rieffel, is to view $\aff(K)$ as a space of quantum observables of the metric space $K$. In light of the work developed in \cite{Kerr:2009aa}, there is likely some scope for a combination of these two approaches.

Ultimately, this article is in support of an earthy philosophy of mathematics. As pointed out in \cite[Proposition 1.1]{Rieffel:2002aa}, there are artificial ways of endowing separable $\cs$-algebras with Lipschitz seminorms. But the highly structured $\cs$-algebras that are the focus of the classification programme are not descended from the heavens. Rather, they have evolved organically from the ground up into diverse populations (inductive limits, crossed products, graph algebras, \textellipsis) whose close relationships sometimes only become apparent via the deep molecular analysis provided by modern classification theory \cite{Carrion:2024aa,Carrion:wz}. Chances are that when we think of the homeomorphism class of an interval, we imagine, well, an interval. Or perhaps not. Maybe my interval is a rigid rod but yours, a vibrating piano string, has time-dependent geometry. And so it is with a classifiable $\cs$-algebra. Chances are that you or I may have a favourite model construction that perhaps in some way is naturally geometric. It is the intention of this article to provide an abstract framework for remembering these geometric origins and measuring and comparing their quantum effects.

\subsection*{Organisation} In Section~\ref{section:cqms}, we recall the definition and some features of compact quantum metric spaces, and introduce our main objects of study (see Definitions~\ref{def:qmcs},~\ref{def:cqmcs}~and~\ref{def:qmbs}) and basic examples (see Example~\ref{ex:cqmcs}). Then in Section~\ref{section:qgh}, we investigate the metric space $(\mathbb{QMCS},\dq)$ of isometry classes of quantum metric Choquet simplices equipped with Rieffel's quantum Gromov--Hausdorff distance (or the quantum intertwining gap defined in Definition~\ref{def:quig}) and prove Theorem~\ref{thm1} (see Corollary~\ref{cor:quig}). In Section~\ref{section:qcp}, we study tracial quantum systems (including random systems associated with Markov--Feller processes in Section~\ref{subsection:random}), build new examples via tracial quantum crossed products (see Corollary~\ref{cor:dynqmcs}) and prove Theorem~\ref{thm3} (see Theorem~\ref{thm:range}). Finally, in Section~\ref{section:cf} we analyse the behaviour of our constructions under continuous metric deformation (see Theorem~\ref{thm:metricdeform}) or continuous parameter variation, and present the proof of Theorem~\ref{thm2} (see Theorem~\ref{thm:cf}).

\subsection*{Acknowledgements} This work is dedicated to the former and future members of Prague's NCG\&T group, to our family and friends, to our fathers and forefathers. My research is supported by the Czech Science Foundation (GA\v{C}R) project 22-07833K and the Institute of Mathematics of the Czech Academy of Sciences (RVO: 67985840). I am grateful to Hanfeng Li and Stuart White for helpful discussions at the Fields Institute during the Thematic Program on Operator Algebras and Applications in the autumn of 2023, and to the members of the Operator Algebra group at the Southern University of Denmark for their kind hospitality and useful feedback during my visit to Odense in February 2024.

\begin{notation} \label{notation}
We adopt the following notation and terminology throughout the article.

\begin{itemize}
\item The set of nonzero positive integers is denoted by $\nn$, and the set $\nn\cup\{0\}$ is denoted by $\nn_0$.
\item If $(E,\|\cdot\|)$ is a normed space, then for $r\ge0$, $B_r(E)$ denotes the ball of radius $r$ centred at $0$, that is,
\[
B_r(E) = B_r(E,\|\cdot\|) = \{x\in E \mid \|x\|\le r\}.
\]
\item If $A$ is a $\cs$-algebra, then $A_{sa}$ denotes the set of self-adjoint elements of $A$, $A_+$ denotes its positive cone, and $\Aut(A)$ and $\End(A)$ denote, respectively, the spaces of $^*$-automorphisms and $^*$-endomorphisms of $A$.
\item A \emph{tracial state} on a $\cs$-algebra $A$ is a positive, norm-one linear functional $\tau\colon A\to\cc$ that satisfies the trace identity $\tau(uau^*)=\tau(a)$ for every $a\in A$ and unitary $u$ in the minimal unitisation $\tilde{A}$ of $A$. We call the collection $T(A)$ of tracial states on $A$ the \emph{trace space of $A$}. It is naturally equipped with the $w^*$-topology (meaning that $\tau_i\to\tau$ if and only if $\tau_i(a)\to\tau(a)$ for every $a\in A$). When $T(A)$ is a Bauer simplex (see Section~\ref{subsection:qmbs}), $\alpha\in\End(A)$ is occasionally in this article called \emph{tracially nondegenerate} if the induced map $T(\alpha)$ preserves the  boundary $\partial_e(T(A))$ (and hence also $T(A)$ itself).
\item A $\cs$-algebra $A$ is said to be \emph{classifiable} if it is simple, separable, nuclear, $\js$-stable and satisfies the UCT. 
\item A stably finite $\cs$-algebra $A$ is said to be \emph{$K$-connected} if the ordered $K_0$-group of the minimal unitisation of $A$ admits a unique state.
\item For a compact Hausdorff space $X$, we identify $T(C(X))$ with the space $\prob(X)$ of Borel probability measures on $X$. Given $\mu\in\prob(X)$ and a continuous map $h\colon X\to X$, $h_*\mu$ denotes the pushforward measure $\mu\circ h^{-1}$.
\item Given a metric $\rho$ on $X$, $\lip(X,\rho)$ denotes the $1$-Lipschitz functions $X\to\rr$, that is,
\[
\lip(X,\rho) = \{f\colon X\to\rr \mid |f(x)-f(y)|\le\rho(x,y) \text{ for every } x,y\in X\}.
\]
The \emph{radius} of $(X,\rho)$ is $r_{(X,\rho)} = \max_{x,y\in X}\frac{\rho(x,y)}{2}$. For $Y\subseteq X$, $Y_\varepsilon$ denotes its \emph{$\varepsilon$-neighbourhood}, that is, $Y_\varepsilon = \bigcup_{x\in Y} B_\varepsilon(x)$, where $B_\varepsilon(x)$ is the open ball of radius $\varepsilon$ centred at $x\in X$.
\item If $K$ is a compact, convex subset of a (real) topological vector space, then $\aff(K)$ denotes the space of of continuous affine maps $K\to\rr$ (where `affine' means that finite convex combinations are preserved). When $K=T(A)$ and $a\in A$, $\widehat{a}$ denotes the element $\tau\mapsto\tau(a)$ of $\aff(T(A))$.
\item If $A$ and $B$ are unital $\cs$-algebras, then $\emb(A,B)$ denotes the set of unital embeddings (that is, injective $^*$-homomorphisms) $A\to B$.
\item If $\gamma$ is an action of a group $G$ on a set $S$, then $S^\gamma$ denotes the set of elements of $S$ that are fixed by $\gamma(g)$ for every $g\in G$. If $\gamma$ is some induced action $\alpha^*$ (for example, the action on $T(A)$ induced by an automorphism of $A$), then we will simply write $S^\alpha$ rather than $S^{\alpha^*}$.
\item The initialisms CQMS, QMCS and QMBS stand for, respectively, `compact quantum metric space', `quantum metric Choquet simplex' and `quantum metric Bauer simplex'. When one of these is adorned with the prefix `$\cs$', this indicates that the origin of the underlying order unit space is $A_{sa}$ for some $\cs$-algebra $A$.
\item If $(E,L)$ is a CQMS, then $r_E=r_{(E,L)}$ denotes the radius of its \emph{state space}
\[
S(E) = \{\text{positive unital linear maps } E\to \rr\}
\]
when equipped with the induced metric $\rho_L$ (see \eqref{eqn:rho}).
\end{itemize}
\end{notation}

\section{Quantum spaces} \label{section:cqms}

In this section, we recall the definition of a compact quantum metric space and introduce quantum metric Choquet simplices.

\subsection{Compact quantum metric spaces} \label{subsection:cqms}

\begin{definition} \label{def:ous}
An \emph{order unit space} consists of a pair $(E,e)$, where $E$ is a real ordered vector space and $e$ is an \emph{order unit} (that is, for every $x\in E$ there exists $r\in\rr$ with $x\le re$) such that the \emph{Archimedean property} holds: if $x\le re$ for every $r\in\rr_+$, then $x\le0$. The associated norm on $E$ is
\begin{equation} \label{eqn:ounorm}
\|x\| = \inf\{r>0 \mid -re \le x\le re\}.
\end{equation}
We refer to $(E,e)$ as a \emph{complete order unit space} if $E$ is complete with respect to the norm \eqref{eqn:ounorm}.
\end{definition}

The example to bear in mind is the self-adjoint part $A_{sa}$ of a unital $\cs$-algebra $A$, with order unit $e=1_A$ and the usual order and norm.

\begin{definition} \label{def:cqms}
A \emph{compact quantum metric space} (CQMS) is an order unit space $(E,e)$ equipped with a \emph{Lip-norm} $L$, that is, a seminorm $E\to[0,\infty]$ such that:
\begin{enumerate}[1.]
\item $\ker L = \rr e$;
\item the topology on the state space $S(E)$ induced by the metric
\begin{equation} \label{eqn:rho}
\rho_L(\sigma,\tau) = \sup\{|\sigma(x)-\tau(x)| \mid x\in E,\:L(x)\le 1\}
\end{equation}
is the $w^*$-topology.
\end{enumerate}
\end{definition}

\begin{remark} \label{rem:lsc}
One difference between Definition~\ref{def:cqms} and Rieffel's definition \cite[Definition 2.2]{Rieffel:2004aa} of a compact quantum metric space is that in \cite{Rieffel:2004aa}, the Lip-norms take finite values. Certain properties (in particular, the notion of an \emph{isometry} between spaces---see Section~\ref{section:qgh}) are then phrased in terms of the norm completion $\overline{E}$ of $E$, on which one defines the \emph{closure} $L^c$ of $L$ as
\begin{equation} \label{eqn:closure}
L^c(y)=\inf\left\{\liminf_{n\to\infty} L(x_n) \mid x_n\in E,\: x_n\to y\right\}.
\end{equation}
\emph{Closed} compact quantum metric spaces, that is, those for which $L=L^c$, are then of particular interest. Since our examples are naturally associated with \emph{complete} order unit spaces, we follow \cite{Kerr:2009aa} and allow Lip-norms to take the value $+\infty$. In this setting, the closure operation replaces $L$ by the largest lower semicontinuous Lip-norm that it dominates. This does not affect the induced metric on the state space (see \cite[Theorem 4.2]{Rieffel:1999aa}).

A second point of divergence between the present treatment and \cite{Rieffel:2004aa} is that lower semicontinuity is a natural feature of all of our of Lip-norms. In fact, the starting point of constructions like Example~\ref{ex:cqmcs}, Proposition~\ref{prop:inherit}, Theorem~\ref{thm:metricdeform} and Theorem~\ref{thm:cf} is a metric $\rho$ inducing the $w^*$-topology on the state space $S(E)$ of a closed order unit space $E$, with $L=L_\rho$ then defined as
\begin{equation} \label{eqn:L}
L_\rho(x) = \sup\left\{\frac{|\sigma(x)-\tau(x)|}{\rho(\sigma,\tau)} \mid \sigma\ne \tau \in S(E)\right\}.
\end{equation}
This formula exhibits $L$ as the pointwise supremum of continuous functions, which implies that $L$ is lower semicontinuous. By \cite[Proposition 3.1]{Li:2006aa} (which is based on \cite[Section 9]{Rieffel:1999aa}), the maps $L\mapsto\rho_L$ and $\rho\mapsto L_\rho$ defined by \eqref{eqn:rho} and \eqref{eqn:L}  are inverse bijections between the set of lower semicontinuous Lip-norms on $E$ and the set of metrics $\rho$ inducing the $w^*$-topology on $S(E)$ that are \emph{convex}, that is, satisfy
\begin{equation} \label{eqn:convex}
\rho\left(\sum_{i=1}^n\lambda_i\sigma_i,\sum_{i=1}^n\lambda_i\tau_i\right) \le \sum_{i=1}^n\lambda_i\rho(\sigma_i,\tau_i)
\end{equation}
for every $n\in\nn$, $\sigma_1,\dots,\sigma_n,\tau_1,\dots\tau_n\in S(E)$ and $\lambda_1,\dots,\lambda_n\in[0,1]$ with $\sum_{i=1}^n\lambda_i=1$, and are also \emph{midpoint-balanced}, meaning that
\begin{equation} \label{eqn:midpt}
\frac{\sigma_1+\tau_2}{2} = \frac{\sigma_2+\tau_1}{2} \implies\rho(\sigma_1,\tau_1)=\rho(\sigma_2,\tau_2)
\end{equation}
for every $\sigma_1,\sigma_2,\tau_1,\tau_2 \in S(E)$. (Without assuming lower semicontinuity, we have $L_{\rho_L}=L^c$; see \cite[Theorem 4.2]{Rieffel:1999aa}.) In particular, if the starting metric $\rho$ has these properties, then $\rho_{L_\rho}=\rho$, so $\rho_{L_\rho}$ does indeed induce the $w^*$-topology on $S(E)$ and therefore $(E,L_\rho)$ is a compact quantum metric space. All in all, when we adapt Definition~\ref{def:cqms} to the tracial setting (Definition~\ref{def:qmcs}), we will add lower semicontinuity to the list of hypotheses.
\end{remark}

\subsection{Quantum metric Choquet simplices} \label{subsection:qmcs}

The key property exhibited by the compact quantum metric spaces $(E,L)$ of interest to us is \emph{Riesz interpolation}, which we recall means that, whenever $x_1,x_2,y_1,y_2$ are elements of $E$ with $x_i\le y_j$ for $i,j\in\{1,2\}$, there exists $z\in E$ such that $x_1,x_2\le z \le y_1,y_2$.

\begin{definition} \label{def:qmcs}
A \emph{quantum metric Choquet simplex} (QMCS) is a compact quantum metric space $(E,L)$ that is norm complete and satisfies the Riesz interpolation property, and such that $L$ is lower semicontinuous and densely finite.
\end{definition}

To explain our choice of terminology in Definition~\ref{def:qmcs}, we first recall the following.

\begin{definition} \label{def:recollection}~
\begin{enumerate}[1.]
\item A metrisable \emph{Choquet simplex} is a compact, convex subset $\Delta$ of a separable, Hausdorff locally convex space whose associated cone is a lattice. Equivalently (see, for example, \cite[Theorem \rm{II}.3.6]{Alfsen:1971hl} or \cite[Chapter 10]{Phelps:2001rz}), every point $x\in\Delta$ is represented by a unique Borel probability measure $\mu_x$ supported on the extremal boundary $\partial_e\Delta$ of $\Delta$ (which in the metrisable setting is Borel), that is, $\mu_x$ is the unique measure on $\Delta$ with $\mu_x(\Delta\setminus\partial_e\Delta)=0$ such that 
\[
f(x) = \int_{\partial_e\Delta} f\,d\mu_x \: \text{ for every } f\in\aff(\Delta).
\]
\item Two elements $a$ and $b$ of the positive cone $A_+$ of a $\cs$-algebra $A$ are \emph{Cuntz--Pedersen equivalent}, written $a\sim_{cp}b$, if there is a sequence $(x_n)_{n\in\nn}$ in $A$ such that $a=\sum_{n\in\nn} x_n^*x_n$ and $b=\sum_{n\in\nn} x_nx_n^*$. Equivalently, $a\sim_{cp}b$ if and only if $\tau(a)=\tau(b)$ for every lower semicontinuous trace $\tau$ on $A$, that is, for every lower semicontinuous, positive-linear map $\tau\colon A^+\to[0,\infty]$ satisfying the trace identity $\tau(x^*x)=\tau(xx^*)$. (This was proved in \cite[Theorem 7.5]{Cuntz:1979fv} assuming simplicity of $A$, and in full generality in \cite{Robert:2009aa}.) The set
\begin{equation} \label{eqn:A0}
A_0 = \left\{a-b \mid a,b\in A_+,\: a\sim_{cp} b\right\}
\end{equation}
is a closed subspace of $A_{sa}$, and the dual of the quotient $A^q=A_{sa}/A_0$ is isometrically isomorphic to the set of bounded tracial self-adjoint functionals on $A$ (see \cite[Propositions 2.7, 2.8]{Cuntz:1979fv}). If $A$ is algebraically simple (in particular, if $A$ is simple and unital) and $a,b\in A_+$, then $a\sim_{cp}b$ if and only if $a-b\in A_0$ (see \cite[Theorem 5.2]{Cuntz:1979fv}, and note that the example of the compact operators demonstrates that this is not true in general).

Since $A_0$ is an order ideal of $A_{sa}$ (that is, if $a,b\in A_0$ and $c\in A_{sa}$ satisfy $a\le c\le b$, then $c\in A_0$), $A^q$ is an ordered vector space when equipped with the quotient ordering (see \cite[Proposition \rm{II}.1.1]{Alfsen:1971hl}), which means that its positive cone is defined to be $A^q_+=q(A_+)$, where $q\colon A_{sa}\to A^q$ denotes the quotient map.
\end{enumerate}
\end{definition}

\begin{definition} \label{def:to}
We say that $A^q$ is \emph{tracially ordered} if $T(A)\ne\emptyset$ and
\begin{equation} \label{eqn:to}
A^q_+ = \{a+A_0 \mid a\in A_{sa},\: \tau(a)\ge0 \text{ for every } \tau\in T(A)\}.
\end{equation}
\end{definition}

Recall that a unital $\cs$-algebra $A$ has \emph{continuous trace} if, for every $a\in A$, the function $\pi \mapsto \tr(\pi(a))$ is finite and continuous on the spectrum $\widehat{A}$ of $A$  (see \cite[Definition 4.5.2]{Dixmier:1964rt}). Here, $\tr$ denotes the trace on $\mathcal{B}(H)$ and part of the definition is that $\pi(a)$ is a trace-class operator for every $\pi\in\widehat{A}$. In fact, by \cite[Proposition 4.5.3]{Dixmier:1964rt}, $A$ is liminal (that is, every irreducible representation is contained in the compact operators) and its spectrum is Hausdorff. By \cite[Theorem 10.5.4]{Dixmier:1964rt}, $A$ is isomorphic to the $\cs$-algebra of sections of the associated continuous field of matrix algebras over $X:=\widehat{A}$. We let $\lambda_A$ denote the corresponding multiplier map from $C(X)$ to the centre of the multiplier algebra $M(A) \cong A$. Fibrewise, $\lambda_A(f)(x)=f(x)\cdot1_{n_x}$, where $n_x$ is the size of the matrix fibre at $x\in X$.

\begin{proposition} \label{prop:to}
Let $A$ be a $\cs$-algebra with $T(A)\ne\emptyset$. Suppose that either
\begin{enumerate}[(i)]
\item \label{it:to1} $A$ is algebraically simple, or
\item \label{it:to2} $A$ is a unital continuous-trace $\cs$-algebra such that the multiplier map $\lambda_A$ induces an affine homeomorphism $T(A) \cong T(C(X))$ (equivalently, every $\tau\in T(A)$ is of the form $\tau(a) = \int_X \tr(a(x))\,d\mu(x)$ for some $\mu \in \prob(X)$, where $a(x)$ denotes the image of $a\in A$ in the fibre $A_x$ at $x\in X:=\widehat{A}$).
\end{enumerate}
Then, $A^q$ is tracially ordered.
\end{proposition}

\begin{proof}
For \eqref{it:to1}, see \cite[Corollary 6.4]{Cuntz:1979fv}. For \eqref{it:to2}, let $a\in A_{sa}$ such that $\tau(a)\ge0$ for every $\tau\in T(A)$. Equivalently, $\widehat{a}(x) := \tr(a(x)) \ge 0$ for every $x\in X$. Let $b=\lambda_A(\widehat{a}) \in A_+$. Then, for every $\tau = \int_X (\cdot) \,d\mu \in T(A)$ we have
\[
\tau(b) = \int_X \tr(b(x))\,d\mu(x) = \int_X \tr\left(\widehat{a}(x)\cdot 1\right)\,d\mu(x) = \int_X \widehat{a}(x)\,d\mu(x) = \tau(a).
\]
It follows that $a\sim_{cp}b$ and, since $b\ge0$, that $a+A_0=b+A_0\in A^q_+$.
\end{proof}

Note that $T(A)\ne\emptyset$ holds for unital, exact, stably finite $A$ (see \cite[Theorem 1.1.4]{Rordam:2002yu}), condition \eqref{it:to1} of Proposition~\ref{prop:to} holds if $A$ is simple and unital, and condition \eqref{it:to2} is satisfied by, for example, the dimension drop algebras that are the building blocks of the Jiang--Su algebra (see \cite{Jiang:1999hb} and also Example~\ref{ex:cqmcs} below). 

\begin{proposition} \label{prop:quotient}
Let $A$ be a unital $\cs$-algebra such that $A^q_+$ is tracially ordered. Then, $(A^q,q(1))$ is a complete order unit space and the induced order unit norm on $A^q$ agrees with the quotient norm.
\end{proposition}

\begin{proof}
To show that $(A^q,e)$ is an order unit space (where $e=q(1)$), we just have to verify the Archimedean property. But if $x$ is an element of $A^q$ such that $x\le re$ for every $r>0$, then $\tau(x)\le0$ for every $\tau\in T(A)$, so by the assumption that $A^q$ is tracially ordered, we must have $x\le0$.

To prove that the order unit norm $\|\cdot\|_0$ agrees with the quotient norm $\|\cdot\|_q$, we must show that $\|q(a)\|_q\le\|q(a)\|_0$ for every $a\in A_{sa}$ (the reverse inequality being automatic---see \cite[Proposition \rm{II}.1.6]{Alfsen:1971hl}). Suppose that $\|q(a)\|_0\le r$, that is, $-re\le q(a) \le re$. Then, $\sup_{\tau\in T(A)}|\tau(a)|\le r$ and so for any $\varepsilon>0$ there exists $b\in A_{sa}$ such that $a+A_0=b+A_0$ and $\|b\|<r+\varepsilon$ (see \cite[Proposition 2.1]{Carrion:wz}). This implies that $\|q(a)\|_q = \inf\{\|a-c\| \mid c\in A_0\} \le r$, and therefore that $\|q(a)\|_q\le\|q(a)\|_0$.
\end{proof}

Finally, we are in a position to justify our terminology.

\begin{theorem} \label{thm:cqmcs}
The following are equivalent for a complete order unit space $E$:
\begin{enumerate}[(i)]
\item \label{it:qmcs1} $E$ is separable and has the Riesz interpolation property;
\item \label{it:qmcs2} $E$ is isomorphic to $\aff(\Delta)$ for some metrisable Choquet simplex $\Delta$;
\item \label{it:qmcs3} $E$ is isomorphic to the Cuntz--Pedersen quotient $A^q$ for some separable, unital $\cs$-algebra $A$ such that $A^q$ is tracially ordered.
\end{enumerate}
\end{theorem}

\begin{proof}
For the equivalence of \eqref{it:qmcs1} and \eqref{it:qmcs2}, see \cite[Corollary \rm{II}.3.11]{Alfsen:1971hl}. To complete the proof of the theorem, we will show that \eqref{it:qmcs2} is equivalent to \eqref{it:qmcs3}.

First, suppose that \eqref{it:qmcs3} holds. By Proposition~\ref{prop:quotient}, $A^q$ is a complete order unit space when equipped with the quotient ordering, and the order unit norm is equal to the quotient norm. By \cite[Propositions 2.7, 2.8]{Cuntz:1979fv}, the dual of $A^q$ is isometrically isomorphic to the Banach space of tracial bounded linear functions on $A$. Under this isomorphism, the state space of the order unit space $A^q$ (that is, the set of unital bounded linear functionals $\varphi\colon A^q\to\rr$ with $\|\varphi\|=1$) corresponds to the set $T(A)$ of tracial states on $A$. By Kadison's representation theorem \cite[Theorem \rm{II}.1.8]{Alfsen:1971hl}, $A^q$ is therefore isomorphic as an order unit space to $\aff(T(A))$ (which is equipped with the ordinary ordering: $f\ge0$ if and only if $f(\tau)\ge0$ for every $\tau\in T(A)$). Since $T(A)$ is a metrisable Choquet simplex (see \cite{Thoma:1964aa} or \cite[Theorem 3.1]{Pedersen:1969kq}), this gives us \eqref{it:qmcs2}.

Suppose conversely that \eqref{it:qmcs2} holds. Let $A$ be a simple, separable, unital $\cs$-algebra such that the metrisable Choquet simplex $\Delta$ is affinely homeomorphic to $T(A)$. (Candidates for $A$ exist within various classes of interest to the Elliott classification programme; see, for example, \cite[Theorem 3.10]{Blackadar:1980zr}, \cite[Theorem 3.9]{Thomsen:1994qy}, \cite[Theorem 4.5]{Jiang:1999hb} and \cite[Theorem 14.10]{Gong:2020ud}.) Since $A^q$ is tracially ordered (by Proposition~\ref{prop:to}\eqref{it:to1}), it follows exactly as above that $A^q$ is isomorphic as an order unit space to $\aff(\Delta)$, which by assumption is isomorphic to $E$, so we have \eqref{it:qmcs3}.
\end{proof}

In light of Theorem~\ref{thm:cqmcs}\eqref{it:qmcs3}, we do not lose any generality in restricting attention from general order unit spaces to $\cs$-algebras, and indeed this is our primary source of examples. We formalise this as follows (cf.\ \cite[Definition 2.3]{Li:2009aa}).

\begin{definition} \label{def:cqmcs}
A \emph{$\cs$-algebraic quantum metric Choquet simplex} ($\cs$-QMCS) is a separable, unital $\cs$-algebra $A$ equipped with a seminorm $L\colon A\to[0,\infty]$ such that:
\begin{enumerate}[1.]
\item $A^q$ is tracially ordered;
\item $L$ is self-adjoint (that is, $L(a^*)=L(a)$ for every $a\in A$), lower semicontinuous and densely finite;
\item $\ker L=\Span_\cc(\{1\}\cup A_0)$;
\item the metric defined by
\[
\rho_L(\sigma,\tau) = \sup\{|\sigma(a)-\tau(a)| \mid a\in A_{sa},\:L(a)\le 1\}
\]
induces the $w^*$-topology on $T(A)$.
\end{enumerate}
\end{definition}

\begin{remark} \label{rem:cqmcs}
There are some practical aspects of Definition~\ref{def:cqmcs} that are worth commenting on.
\begin{enumerate}[1.]
\item By Proposition~\ref{prop:quotient}, $A^q$ is a complete order unit space when equipped with the quotient ordering and norm. The assumptions of lower semicontinuity of $L$ and density of $E=L^{-1}([0,\infty))$ as a subset of $A_{sa}$ moreover ensure that $E/A_0$ is a closed compact quantum metric space in the sense of Rieffel (see Definition~\ref{def:cqms} and Remark~\ref{rem:lsc}) whose norm completion is $A^q$. Going forward, this will help us to transfer the theory and structure of compact quantum metric spaces to the setting of tracial $\cs$-algebras (see, for example, Theorem~\ref{thm:nucleus}).
\item Occasionally, for example in Section~\ref{section:qcp}, it is useful to allow for nonunital $\cs$-algebras $A$. Of particular interest to us are certain stably projectionless classifiable $\cs$-algebras, namely the $\js_0$-stable ones mentioned in Section~\ref{section:intro} and discussed further in Section~\ref{subsection:range}. In general, such an $A$ can admit unbounded traces. This is not an issue in Section~\ref{section:qcp} because we will use $\cs$-algebras $A$ that have `continuous scale' in the sense of \cite[Section 5]{Elliott:2020vp}, which is a property that for simple, $\js$-stable $\cs$-algebras $A$ is equivalent to \emph{algebraic} simplicity of $A$ and compactness of $T(A)$ (see \cite[Theorem 5.3]{Elliott:2020vp}). In particular, for such an $A$, all traces are bounded (because $A$ coincides with its Pedersen ideal) and $A^q$ still has the structure of a complete order unit space isomorphic to $\aff(T(A))$. Note in particular that, as in \cite[Theorem 9.3]{Lin:2007qf}, there is a positive element $e\in A$ with $\tau(e)=1$ for every $\tau\in T(A)$, and then $q(e)$ serves as an order unit. In these well-behaved cases, we say that $(A,L)$ is a nonunital $\cs$-QMCS if $L\colon A\to[0,\infty]$ is a seminorm satisfying the requirements of Definition~\ref{def:cqmcs} with the unit of $A$ replaced by such an $e$.
\end{enumerate}
\end{remark}

The following version of \cite[Lemma 4.1]{Li:2006aa} provides a basis for metric analysis in the set $\emb(A,B)$ of unital embeddings of a quantum metric Choquet simplex $A$ into unital $\cs$-algebras $B$ (or rather, in the set of approximate unitary equivalence classes $\emb(A,B)/\sim_{au}$).

\begin{theorem} \label{thm:nucleus}
Let $(A,L)$ be a $\cs$-algebraic quantum metric Choquet simplex. Then, for every $r\ge r_{(A,L)}$, there exists a compact set $\mathcal{D}_r(A) \subseteq A_{sa}$ such that
\[
q(\mathcal{D}_r(A)) = \{x\in A^q \mid L(x)\le 1,\|x\|\le r\}
\]
and
\begin{equation} \label{eqn:nucleus}
\mathcal{D}_r(A) + \rr1 + A_0 = \{a\in A_{sa} \mid L(a)\le 1\}.
\end{equation}
We call such a set $\mathcal{D}_r(A)$ a \emph{nucleus of $(A,L)$}.
\end{theorem}

\begin{proof}
The Bartle--Graves theorem (noted by Michael as a corollary to \cite[Theorem 3.2'']{Michael:1956aa}) says that every continuous linear surjection between Banach spaces admits a continuous right inverse. Applying this to the quotient map $q \colon A_{sa} \to A^q$, there is a continuous function $f\colon A^q \to A_{sa}$ such that $f(x) \in q^{-1}(x)$ for every $x \in A^q$. Let $r\ge r_{(A,L)}:= \frac{1}{2}\max\{\rho_L(\sigma,\tau) \mid \sigma,\tau\in T(A)\}$. By \cite[Proposition 2.3]{Li:2006aa} (a restatement of \cite[Theorem 1.9]{Rieffel:1998ww}), the set $\mathcal{D}^q_r(A):=\{x\in A^q \mid L(x)\le 1,\: \|x\|_q\le r\}$ is totally bounded, hence compact by completeness of $A^q$, and moreover
\[
\mathcal{D}^q_r(A)+\rr q(1)=\{q(a) \mid a\in A_{sa}, L(a)\le1\}.
\]
The image $f(\mathcal{D}^q_r(A))$ of $\mathcal{D}^q_r(A)$ under Michael's selection is the stipulated compact set $\mathcal{D}_r(A)$.
\end{proof}

\subsection{Quantum metric Bauer simplices} \label{subsection:qmbs}

The classical example of a compact quantum metric space (and indeed a quantum metric Choquet simplex) is the commutative one. If $(X,\rho)$ is a compact metric space, then the associated Lipschitz seminorm
\[
L_\rho(f) = \sup\left\{\frac{|f(x)-f(y)|}{\rho(x,y)} \mid x\ne y \in X\right\}
\]
is a Lip-norm on $C(X)_{sa}$, and $(C(X),L_\rho)$ is a $\cs$-QMCS.

One special feature of this example is that the trace space is a \emph{Bauer simplex}, that is, a Choquet simplex $\Delta$ whose boundary $X=\partial_e\Delta$ is compact. In this situation, the map that sends a point $\tau$ to its representing measure $\mu_\tau$ is an affine homeomorphism from $\Delta$ to $\prob(\partial_e\Delta)$ (see Section~\ref{subsection:qmcs} and \cite[Theorem \rm{II}.4.1]{Alfsen:1971hl}). In fact, for $E=C(X)_{sa}$ and $\Delta=S(E)$, this map is an \emph{isometry} between $(S(E),\rho_{L_\rho})$ and $(\prob(\partial_eS(E)),W_\rho)$, where $\rho_{L_\rho}$ is defined by \eqref{eqn:rho}, that is,
\[
\rho_{L_\rho}(\sigma,\tau) = \sup\{|\sigma(x)-\tau(x)| \mid x\in E,\:L_\rho(x)\le 1\}
\]
and $W_\rho$ is the \emph{Wasserstein} (or \emph{Monge--Kantorovich} or \emph{Kantorovich--Rubinstein}) metric
\begin{equation} \label{eqn:wass}
W_\rho(\mu,\nu) = \sup\left\{\left|\int_Xf\,d\mu - \int_Xf\,d\nu\right| \mid f\in\lip(X,\rho)\right\}
\end{equation}
associated with $\rho$. As this sort of structure is also the basis for many of our noncommutative examples, we formalise it as follows.

\begin{definition} \label{def:qmbs}
A \emph{quantum metric Bauer simplex} (QMBS) is a quantum metric Choquet simplex $(E,L)$ whose state space $S(E)$ is a Bauer simplex with
\begin{equation} \label{eqn:isometry}
\rho_L(\sigma,\tau) = W_{\rho_L}(\mu_\sigma,\mu_\tau) \:\text{ for every } \sigma,\tau\in S(E).
\end{equation}
A \emph{$\cs$-algebraic quantum metric Bauer simplex} ($\cs$-QMBS) is a $\cs$-algebraic quantum metric Choquet simplex $A$ such that $T(A)$ is a Bauer simplex and such that \eqref{eqn:isometry} holds in $T(A)$.
\end{definition}

\begin{proposition} \label{prop:dirichlet}
Let $E$ be a separable, complete order unit space and $L\colon E \to [0,\infty]$ a densely finite, lower semicontinuous Lip-norm. Then, $(E,L)$ is a quantum metric Bauer simplex if and only if
\begin{enumerate}[(i)]
\item $E$ is a vector lattice and
\item every Lipschitz function $f$ from $(\partial_e(S(E)),\rho_L)$ to $\rr$ admits a (unique) continuous affine extension to $(S(E),\rho_L)$ without increase of uniform norm or Lipschitz seminorm.
\end{enumerate}
In this case,
\begin{equation} \label{eqn:dirichlet}
L(x) = L_{\rho_L}\left(\widehat{x}|_{\partial_eS(E)}\right) \quad \text{for every } x\in E,
\end{equation}
where $\widehat{x}\in\aff(S(E))$ denotes the function $\sigma\mapsto\sigma(x)$.
\end{proposition}

\begin{proof}
By \cite[Theorem \rm{II}.4.1]{Alfsen:1971hl}, the compact, convex set $S(E)$ is a Bauer simplex if and only if $\aff(S(E))$, which by Kadison's representation theorem \cite[Theorem \rm{II}.1.8]{Alfsen:1971hl} is isomorphic to $E$, is a lattice in the usual ordering. In this case, every continuous function $f\colon \partial_e(S(E))\to\rr$ uniquely extends affinely and continuously to $S(E)$ via $\tau\mapsto\int_{\partial_e(S(E))}f\,d\mu_\tau$ (see \cite[Proposition \rm{II}.3.13]{Alfsen:1971hl}). This extension $\tilde{f}$ has the same uniform norm as $f$ (already noted in \cite[Proposition \rm{II}.3.13]{Alfsen:1971hl}), and the same Lipschitz seminorm if \eqref{eqn:isometry} holds: for $\sigma,\tau\in S(E)$ we have
\[
\left|\sigma(\tilde{f}) - \tau(\tilde{f})\right| = \left|\int f\,d\mu_\sigma - \int f\,d\mu_\tau\right| \le L_{\rho_L}(f)W_{\rho_L}(\mu_\sigma,\mu_\tau) = L_{\rho_L}(f) \rho_L(\sigma,\tau),
\]
so
\[
L_{\rho_L}\big(\tilde{f}\big) \le L_{\rho_L}\big(f\big) = L_{\rho_L}\big(\tilde{f}|_{\partial_e(S(E))}\big) \le L_{\rho_L}\big(\tilde{f}\big),
\]
giving $L_{\rho_L}\big(\tilde{f}\big) = L_{\rho_L}\big(f\big)$.

Conversely, if every $f\colon \partial_e(S(E))\to\rr$ can be extended in the manner asserted, then for every $\sigma,\tau\in S(E)$,
\begin{align*}
W_{\rho_L}(\mu_\sigma,\mu_\tau) &= \sup\left\{\left|\int f\,d\mu_\sigma - \int f\,d\mu_\tau\right| \mid f\in\lip(\partial_e(S(E)),\rho_L)\le1\right\}\\
&= \sup\left\{\left|\sigma(f) - \tau(f)\right| \mid f\in E,\,L(f)\le1\right\}\\
&= \rho_L(\sigma,\tau).
\end{align*}
Equation \eqref{eqn:dirichlet} holds because of uniqueness of the extension from the boundary, and indeed we could have equivalently used it in place of \eqref{eqn:isometry} in Definition~\ref{def:qmbs}.
\end{proof}

\begin{example} \label{ex:cqmcs}
As in \cite{Jacelon:2021vc,Jacelon:2022wr} (although the language of quantum spaces was not used there), our foundational examples of $\cs$-algebraic quantum metric Choquet simplices arise as inductive limits of subhomogeneous building blocks. 
\begin{enumerate}[1.]
\item \label{it:ex1} Let $(X,\rho)$ be a compact metric space and $n\in\nn$, and let $A=C(X,M_n)$ (so that, in particular, $\partial_e(T(A)) \cong X$). Define $L\colon A\to[0,\infty]$ by
\begin{equation} \label{eqn:classical}
L(f) = \sup\left\{\frac{|\tr(f(x))-\tr(f(y))|}{\rho(x,y)} \mid x\ne y \in X\right\}.
\end{equation}
Then, $(A,L)$ is a quantum metric Bauer simplex with $\rho_L=W_\rho$ (the Wasserstein metric on $\prob(X)\cong T(A)$). Moreover, we can make an explicit choice of the nucleus $\mathcal{D}_r(A)$ of Theorem~\ref{thm:nucleus}, namely, for $r\ge r_{(X,\rho)}$ (the $\rho$-radius of $X$) we set
\begin{equation} \label{eqn:drmat}
\mathcal{D}_r(A) = \{f \in A_{sa} \mid \|f\|\le r,\: \|f(x)-f(y)\|\le\rho(x,y)\:\text{ for every } x,y\in X\}.
\end{equation}
Then, the $\rr$-linear span of $\mathcal{D}_r(A)$ is dense in $A_{sa}$ (a fact which follows from the Stone--Weierstrass theorem and in particular shows that $L$ is densely finite), and \eqref{eqn:nucleus} holds. In one direction, norm Lipschitz elements are tracially Lipschitz, since for every $x,y\in X$ and $f\in A_{sa}$,
\[
|\tr(f(x))-\tr(f(y))| \le \max_{1\le i\le n} |\lambda_i(x)-\lambda_i(y)| \le \|f(x)-f(y)\|,
\]
where $\lambda_1(t)\le\dots\le\lambda_n(t)$ denote the eigenvalues of $f(t)$ for $t\in X$. In the other direction, suppose that $L(f)\le 1$, so that $\widehat{f}\in\lip(X,\rho)$ (where $\widehat{f}(x):=\tr(f(x))$). Set $g:=\widehat{f}\cdot1_n$ and choose $x_0,x_1\in X$ such that $\rho(x_0,x_1) = r_{(X,\rho)}$. Then, $g - g(x_0) \in\mathcal{D}_r(A)$ and, since $\tr(g(x))=\tr(f(x))$ for every $x\in X$, we also have $f-g\in A_0$. In other words, $f \in \mathcal{D}_r(A) + \rr1 + A_0$.\\

\noindent Note that this representative $g$ of $f$ in $A^q$ is in fact the multiplier element $g=\lambda_A(\widehat{f})$ of Proposition~\ref{prop:to}. In this context, the multiplier map $C(X) \to A$ does not just provide an affine homeomorphism of trace spaces but, by design, an affine \emph{isometry}. This structure also passes to unital subalgebras like \emph{dimension drop algebras}, which we recall are $\cs$-algebras of the form
\[
Z_{p,q} = \left\{f\in C([0,1],M_p\otimes M_q) \mid f(0)\in M_p \otimes 1_q,\: f(1)\in 1_p \otimes M_q\right\}, \ p,q \in \nn,
\]
and which do indeed satisfy hypothesis \eqref{it:to2} of Proposition~\ref{prop:to}. The nucleus \eqref{eqn:drmat} is used in \cite{Jacelon:2021vc} as a basis for measurement of the distance between the unitary orbits of unital embeddings of (certain inductive limits of) dimension drop algebras into classifiable $\cs$-algebras (see Section~\ref{subsection:metric}).

\item \label{it:ex2} As for \emph{classifiable} examples, \cite[Remark 4.5]{Jacelon:2022wr} indicates that for any nonempty compact metric space $(X,\rho)$, there is a unital, classifiable $\cs$-algebra $\mathcal{A}_X$ with $\partial_e(T(\mathcal{A}_X))\cong X$ such that $\rho$ endows $\mathcal{A}_X$ with the structure of a $\cs$-QMBS via \eqref{eqn:dirichlet}. In other words, the seminorm $L\colon \mathcal{A}_X \to [0,\infty]$ is defined by 
\begin{equation} \label{eqn:lbauer}
L(a) = L_\rho(\widehat{a}) = \sup\left\{\frac{|\widehat{a}(x)-\widehat{a}(y)|}{\rho(x,y)} \mid x\ne y \in X\right\}
\end{equation}
where, as usual, $\widehat{a}(\tau)=\tau(a)$ for $\tau\in T(\mathcal{A}_X)$. The $\cs$-algebra $\mathcal{A}_X$ is built as an inductive limit of homogeneous building blocks $C(X,M_{n_i})$ with diagonal connecting maps, most of whose eigenvalue functions are equal to the identity map $X\to X$ (and the rest are constant). If $X$ is connected, then $\mathcal{A}_X$ is $K$-connected, and we can arrange for the image $\rho_{\mathcal{A}_X}(K_0(\mathcal{A}_X))$ of the pairing map to be any prescribed dense subgroup of $\qq\cdot1\subseteq\aff(T(\mathcal{A}_X))$. If $X$ is a compact, connected Riemannian manifold and $\rho$ is its intrinsic metric, then we can modify the construction of $\mathcal{A}_X$ as in \cite[Theorem 4.4]{Jacelon:2022wr} to produce a \emph{projectionless} model $\js_X$. These are built as inductive limits of \emph{generalised} dimension drop algebras
\[
X_{p,q} = \left\{f\in C(X,M_p\otimes M_q) \mid f(x_0)\in M_p\otimes 1_q,\: f(x_1)\in1_p\otimes M_q\right\}
\]
with connecting maps carefully controlled to be compatible with the metric structure so that tracially Lipschitz elements at finite stages are mapped to tracially Lipschitz elements in the limit (a task aided by the manifold's abundance of geodesics).

\item \label{it:ex3} Suppose that $A$ is a simple, unital $\cs$-algebra that is isomorphic to an inductive limit $\varinjlim(A_i,\varphi_i)$ of homogeneous building blocks $A_i=C(X_i,M_{n_i})$ over compact metric spaces $(X_i,\rho_i)$ of uniformly bounded diameter. Let us show that $A$ can be equipped with $\cs$-QMCS structure in such a way that the Lipschitz elements in each $A_i$ are mapped to Lipschitz elements in $A$. We equip each $A_i$ with the corresponding $\cs$-QMBS structure as in Example~\ref{it:ex1}, which equivalently means that we equip $T(A_i)\cong\prob(X_i)$ with the Wasserstein metric $W_{\rho_i}$ given by \eqref{eqn:wass}. The trace space  $T(A)$ (as usual with the $w^*$-topology) is affinely homeomorphic to the inverse limit $\varprojlim(T(A_i),\varphi_i^*)$ (which is equipped with the product topology inherited from $\prod_{i\in\nn} T(A_i)$). We then define a metric $\rho_\Sigma$ on $T(A)$ by
\[
\rho_\Sigma((\sigma_i)_{i=1}^\infty,(\tau_i)_{i=1}^\infty) := \sum_{i=1}^\infty \omega_iW_{\rho_i}(\sigma_i,\tau_i)
\]
for some fixed summable sequence of real numbers $\omega_i>0$. For any $j$, we have by definition that $\rho_\Sigma((\sigma_i),(\tau_i))\ge\omega_jW_{\rho_j}(\sigma_j,\tau_j)$, so the induced map $(T(A),\rho_\Sigma)\to (T(A_j),W_{\rho_j})$ is $\omega_j^{-1}$-Lipschitz, so the Lipschitz seminorm of any $a\in A_j$ is scaled by at most $\omega_j^{-1}$ in the limit. This in particular implies that tracially Lipschitz elements are dense in $A$. Moreover, $\rho_\Sigma$ induces the $w^*$-topology on $T(A)$, and is convex \eqref{eqn:convex} and midpoint-balanced \eqref{eqn:midpt} since each $W_{\rho_i}$ is. As explained in Remark~\ref{rem:lsc} (here applied to $E=A^q$), this implies that $(A,L_{\rho_\Sigma})$ is a $\cs$-QMCS.

Note that in this construction, $T(A)$ need not be a Bauer simplex and $A$ need not be classifiable. Included as examples are the non-$\js$-stable Villadsen algebras whose trace spaces are shown in \cite{Elliott:2023aa} to be affinely homeomorphic to the Poulsen simplex (characterised by density of its extreme points, so, very much not a Bauer simplex). 
\end{enumerate}
\end{example}

In Section~\ref{section:qcp}, we will discover more examples via dynamics. In fact, we will see that any metrisable Choquet simplex can be exhibited as the trace space of a classifiable and \emph{dynamical} $\cs$-QMCS, that is, one constructed as a tracial quantum crossed product associated with a dynamical system on a (zero-dimensional) compact metric space (see Definition~\ref{def:qcp} and Corollary~\ref{cor:dynqmcs}).

\subsection{Metric analysis} \label{subsection:metric}

Finally, we describe two examples of what we mean by `metric analysis' in $\emb(A,B)$.

\subsubsection{Unitary distances} \label{subsubsection:unitary}

Let $(A,L)$ be a $\cs$-QMCS and $B$ a unital $\cs$-algebra. The \emph{unitary distance} in $\emb(A,B)$ relative to a nucleus $\mathcal{D}_r(A)$ is defined to be
\begin{equation} \label{eqn:unitary}
d_U(\varphi,\psi)|_{\mathcal{D}_r(A)} := \inf_{u\in\mathcal{U}(B)} \sup_{a\in\mathcal{D}_r(A)} \|\varphi(a)-u\psi(a)u^*\|.
\end{equation}
This is a reasonable measure of distance in the following sense. On the one hand, if $\varphi\sim_{au}\psi$ (that is, there is a sequence of unitaries $(u_n)_{n\in\nn}$ in $B$ such that $\varphi(a)=\lim_{n\to\infty}u_n\psi(a)u_n^*$ for every $a\in A$), then $d_U(\varphi,\psi)|_{\mathcal{D}_r(A)}=0$ by compactness of $\mathcal{D}_r(A)$. On the other hand, if $d_U(\varphi,\psi)|_{\mathcal{D}_r(A)}=0$, then $T(\varphi)|_{\mathcal{D}_r(A)}=T(\psi)|_{\mathcal{D}_r(A)}$ (where $T(\varphi)$ and $T(\psi)$ are the continuous affine maps $T(B)\to T(A))$ induced by $\varphi$ and $\psi$), so by \eqref{eqn:nucleus} and density of $L^{-1}([0,\infty))$ in $A$, we have $T(\varphi)=T(\psi)$. In particular, if $B$ is classifiable and has real rank zero (so that the $\overline{K}_1^{\mathrm{alg}}$-component of the classification \cite[Theorem B]{Carrion:wz} disappears), then $\varphi\sim_{au}\psi$ if and only if $KL(\varphi)=KL(\psi)$ (which, if $K_*(A)$ is finitely generated and torsion free, is equivalent to $K_*(\varphi)=K_*(\psi)$). In this case, $d_U(\cdot,\cdot)|_{\mathcal{D}_r(A)}$ is a metric on $K$-localised components of $\emb(A,B)/\sim_{au}$ in the sense of \eqref{eqn:klocal}. In some `near-classical' cases (see Example~\ref{ex:cqmcs}.\ref{it:ex1}), it is possible to find a nucleus $\mathcal{D}_r(A)$ whose linear span is dense in $A_{sa}$, and which therefore affords the more satisfying conclusion $d_U(\varphi,\psi)=0 \implies \varphi\sim_{au}\psi$ without appeal to $K$-theory or classification. The question of a measure-theoretic computation of $d_U$ in such cases is considered in \cite{Jacelon:2021wa,Jacelon:2021vc} and will not be taken up further here (but see Example~\ref{ex:metricdeform2}).

\subsubsection{Birkhoff convergence rates} \label{subsubsection:birkhoff}

Let $(A,L)$ be a $\cs$-QMBS. As in \cite[Section 3]{Jacelon:2022wr}, we call a $^*$-homomorphism $\alpha \colon A \to A$ \emph{tracially nondegenerate} if $\alpha$ induces a continuous affine map $T(\alpha) \colon T(A) \to T(A)$ that preserves the boundary $\partial_e(T(A))$.

Suppose that $\alpha$ is such a $^*$-endomorphism with a unique fixed trace, that is, there is a unique $\tau_\alpha\in T(A)$ such that $\tau_\alpha\circ\alpha=\tau_\alpha$. Equivalently, writing $X=\partial_e(T(A))$ and $h=T(\alpha)|_{X}$, there is a unique $\mu=\mu_{\tau_\alpha}\in\prob(X)$ such that $h_*\mu=\mu$. In this setting, Birkhoff's ergodic theorem tells us that, for every $f\in C(X)$, the time average $\frac{1}{n}\sum_{k=0}^{n-1} f\circ h^k$ converges uniformly to the spatial average $\int_Xf\,d\mu$. For every $r\ge r_{(A,L)} = r_{(X,\rho_L)}$, the convergence is also uniform over the compact set
\begin{equation} \label{eqn:drcomm}
\mathcal{D}_{r}(C(X))=\lip(X,\rho_L)\cap B_{r}(C(X)).
\end{equation}
Using the fact that every $\tau\in T(A)$ can be approximated in the $w^*$-topology by a finite convex combination of extremal traces, we have the following interpretation of Birkhoff's theorem at the level of the $\cs$-algebra: for every $r \ge r_{(A,L)}$ and $\varepsilon>0$, there exists $N\in\nn$ such that for every $n\ge N$,
\begin{equation} \label{eqn:birkhoff}
\sup_{a\in\mathcal{D}_r(A)} \sup_{\tau\in T(A)} \left| \frac{1}{n}\sum_{k=0}^{n-1} \tau(\alpha^k(a))-\tau_\alpha(a)\right| \le \varepsilon.
\end{equation}

Lest one think that we are making minimal use of the ambient quantum metric structure, we recall that the Arzel\`{a}--Ascoli theorem characterises compactness in $C(X)$ via uniform boundedness and equicontinuity, which of course are the properties that define $\mathcal{D}_r(A)$. In Section~\ref{subsection:random}, we will observe that a version of \eqref{eqn:birkhoff} holds for \emph{random} tracially nondegenerate endomorphisms of $A$ associated with Markov--Feller processes on $\partial_e(T(A))$.

\begin{definition} \label{def:birkhoff}
Let $(A,L)$, $\alpha$ and $r$ be as above. The \emph{Birkhoff convergence rate} of $(A,L,\alpha)$ is the function $\beta_r\colon (0,\infty) \to \nn$, where $\beta_r(\varepsilon)$ is defined to be the least $N$ such that \eqref{eqn:birkhoff} holds for every $n\ge N$.
\end{definition}
In Section~\ref{section:cf}, we investigate how this convergence rate varies as we change $L$. For now, we point out an effect that we metaphorically associate with gravitational time dilation. Viewing time as the perception of dynamical evolution (perhpas planetary rotation or atomic decay in reality, or Birkhoff convergence of an endomorphism in $\cs$-reality) and mass as embodied by the size of the nucleus $\mathcal{D}_r(A)$, the metric-warping effect of an increase in mass is a corresponding decrease of the Lip-norm $L$, and consequently a lengthening of the metric $\rho_L$ and dilation of time $\beta_r(\varepsilon)$.

\section{Quantum distances} \label{section:qgh}

Recall that the \emph{Hausdorff distance} $\dist^\rho_{\mathrm{H}}(Y,Z)$ (or, depending on what is being emphasised, $\dist^X_{\mathrm{H}}(Y,Z)$) between nonempty subsets $Y$ and $Z$ of a metric space $(X,\rho)$ is
\[
\dist^\rho_{\mathrm{H}}(Y,Z) = \inf\{r>0 \mid Y\subseteq Z_r,\: Z\subseteq Y_r\},
\]
where $S_r$ denotes the $r$-neighbourhood of $S\subseteq X$ as in Notation~\ref{notation}. The \emph{Gromov--Hausdorff distance} between compact metric spaces $(X_1,\rho_1)$ and $(X_2,\rho_2)$ is
\begin{align} \label{eqn:dgh}
\dist_{\mathrm{GH}}(X_1,X_2) = \inf&\{\dist^\rho_{\mathrm{H}}(h_1(X_1),h_2(X_2)) \mid h_1\colon X_1\to X,\,h_2\colon X_2\to X \text{ are} \notag\\ &\text{isometric embeddings into some metric space } (X,\rho)\}
\end{align}
(see \cite{Gromov:1981aa}, and note that it is sufficient to consider embeddings into the disjoint union $X_1 \sqcup X_2$). In \cite{Rieffel:2004aa}, the distance $\dq$ between (isometry classes of) compact quantum metric spaces is defined as a quantum analogue of $\dist_{\mathrm{GH}}$.

\begin{definition} \label{def:qgh}
Let $(E_1,L_1)$ and $(E_2,L_2)$ be compact quantum metric spaces. A Lip-norm $L$ on $E_1\oplus E_2$ is \emph{admissible} if it induces $L_1$ and $L_2$ under the quotient maps $\pi_i\colon E_1\oplus E_2 \to E_i$, that is, $L_i(x) = \inf\{L(y) \mid \pi_i(y)=x\}$ for $i=1,2$ and $x\in E_i$. The \emph{quantum Gromov--Hausdorff distance} between $(E_1,L_1)$ and $(E_2,L_2)$ is
\[
\dq(E_1,E_2) = \inf\{\dist_{\mathrm{H}}^{\rho_L}(S(E_1),S(E_2)) \mid L\, \text{ an admissible Lip-norm on}\, E_1\oplus E_2\}.
\]
An \emph{isometry} between compact quantum metric spaces $(E_1,L_1)$ and $(E_2,L_2)$ is an order isomorphism $\varphi\colon \overline{E_1}\to \overline{E_2}$ such that $L_2^c\circ\varphi=L_1^c$ (where $L_i^c$ is defined as in \eqref{eqn:closure}).
\end{definition}

\begin{proposition}[\cite{Rieffel:2004aa}] \label{prop:isometry}
Let $(E_1,L_1)$ and $(E_2,L_2)$ be compact quantum metric spaces.
\begin{enumerate}[(i)]
\item \label{it:isometry1} If $(E_1,L_1)$ and $(E_2,L_2)$ are closed (see Remark~\ref{rem:lsc}), then the isometries from $(E_1,L_1)$ to $(E_2,L_2)$ are in bijective correspondence with the affine isometries from $(S(E_2),\rho_{L_2})$ onto $(S(E_1),\rho_{L_1})$.
\item \label{it:isometry2} There exists an isometry between $(E_1,L_1)$ and $(E_2,L_2)$ if and only if $\dq(E_1,E_2)=0$.
\end{enumerate}
\end{proposition}

\begin{proof}
For \eqref{it:isometry1}, see \cite[Corollary 6.4]{Rieffel:2004aa}. For the `only if' direction of \eqref{it:isometry2}, see \cite[Proposition 7.1]{Rieffel:2004aa} and its preceding discussion, and for the `if' direction, see \cite[Theorem 7.8]{Rieffel:2004aa}. 
\end{proof}

\begin{definition} \label{def:CQMS}
We let $\mathbb{CQMS}$ denote the set of isometry classes of compact quantum metric spaces equipped with the metric $\dq$. We write $\mathbb{QMCS}$ and $\mathbb{QMBS}$ for the subspaces of isometry classes of quantum metric Choquet simplices and quantum metric Bauer simplices, respectively.
\end{definition}

We note that $\dq$ admits the following equivalent formulation in $\mathbb{QMCS}$ (cf.\ \cite[Proposition 3.2]{Li:2006aa}).

\begin{proposition} \label{prop:gh}
Let $(E_1,L_1)$ and $(E_2,L_2)$ be quantum metric Choquet simplices. Then,
\begin{align*}
\dq(E_1,E_2) = \inf&\{\dist^\rho_{\mathrm{H}}(h_1(S(E_1)),h_2(S(E_2))) \mid h_i\colon S(E_i)\to S(E),\,i=1,2,\\ &\text{are affine isometric embeddings into the state space of a}\\
&\text{quantum metric Choquet simplex $(E,L)$}\}.
\end{align*}
In fact, the infimum can be taken over affine isometric embeddings $h_i\colon S(E_i)\to S(E)$ such that $h_i(\partial_e(S(E_i)))\subseteq \partial_e(S(E))$ for $i=1,2$.
\end{proposition}

\begin{proof}
Denote the quantity on the right-hand side of the displayed equation by $\mathrm{dist}_{\mathrm{q}}'(E_1,E_2)$. If $L$ is an admissible Lip-norm on $E=E_1\oplus E_2$, then the quotient maps $\pi_i\colon E \to E_i$ induce isometric embeddings $h_i\colon S(E_i) \to S(E)$ such that $h_i(\partial_e(S(E_i)))\subseteq \partial_e(S(E))$ and indeed $S(E)=\conv h_1(S(E_1))\cup h_2(S(E_2))$. Moreover, $(E,L)$ is a quantum metric Choquet simplex since the Riesz interpolation property passes to direct sums. Therefore, $\mathrm{dist}_{\mathrm{q}}'(E_1,E_2)\le\dq(E_1,E_2)$. For the reverse inequality,
suppose that $h_i\colon S(E_i)\to S(E)$, $i=1,2$, are affine isometric embeddings. For $i=1,2$, let $\pi_i\colon(E,L)\to(B_i,M_i)$ be the quotients given by restricting the state space to $h_i(S_i)$, that is, $B_i=\aff(h_i(S_i))$, $\pi_i$ is the evaluation map and $M_i(x) = \inf\{L(y) \mid \pi_i(y)=x\}$ for $x\in B_i$. By Proposition~\ref{prop:isometry}\eqref{it:isometry1}, there are isometries between $(B_i,M_i)$ and $(E_i,L_i)$, $i=1,2$, so by Proposition~\ref{prop:isometry}\eqref{it:isometry2}, $\dq(E_1,E_2)=\dq(B_1,B_2)$. But $\dq(B_1,B_2) \le \dist^\rho_{\mathrm{H}}(h_1(S(E_1)),h_2(S(E_2)))$ by \cite[Proposition 5.7]{Rieffel:2004aa}, and so it follows that $\dq(E_1,E_2) \le \mathrm{dist}_{\mathrm{q}}'(E_1,E_2)$.
\end{proof}

In analogy with Gromov's completeness and compactness theorems, it is proved in \cite[Theorem 12.11]{Rieffel:2004aa} that $(\mathbb{CQMS},\dq)$ is complete and in \cite[Theorem 13.5]{Rieffel:2004aa} that a subset $\mathbb{S}\subseteq\mathbb{CQMS}$ is totally bounded if and only if the sets of
\begin{enumerate}[(i)]
\item radii $r_{(E,L)} = r_{\rho_L}(S(E))$, and
\item covering growths
\begin{align}
\mathrm{Cov}_{\rho_L}&\colon(0,\infty) \to \nn\notag\\
\varepsilon &\mapsto \inf\{K\in\nn \mid \text{$(S(E),\rho_L)$ can be covered by $K$ open $\varepsilon$-balls}\} \label{eqn:covering}
\end{align}
\end{enumerate}
of elements $(E,L) \in \mathbb{S}$ are uniformly bounded (by, respectively, some positive real number $R$ and some function $G\colon(0,\infty) \to \nn$). One of the motivating questions for this section is the following.

\begin{question} \label{q:qmcs}
Are the metric spaces $(\mathbb{QMCS},\dq)$ and $(\mathbb{QMBS},\dq)$ complete?
\end{question}

The question for $\mathbb{QMCS}$ is left open, but we will show that $\mathbb{QMBS}$ is indeed complete (Corollary~\ref{cor:complete}) via a quasimetric called the quantum intertwining gap $\gamma_q$ (Definition~\ref{def:quig}) that turns out to be equivalent to $\dq$ (Theorem~\ref{thm:quig}).

The other motivation for this section is the $\cs$-algebraic isometric isomorphism problem. As mentioned in the Introduction, one deficit of the distance $\dq$ is that, if $(A,L)$ and $(B,M)$ are $\cs$-algebraic compact quantum metric spaces, then $\dq(A_{sa},B_{sa})=0$ does not imply that $A\cong B$ as $\cs$-algebras. This can be rectified by, for example, adjusting $\dq$ so that it includes either operator system data (as in \cite{Kerr:2009aa}) or data associated with the Leibniz property (as in \cite{Latremoliere:2016wn}). A pleasing aspect of working with \emph{$K$-connected, classifiable} $\cs$-algebraic quantum metric Choquet simplices is that the isomorphism question then has a $K$-theoretic answer: if $\dq(A^q,B^q)=0$, then by Proposition~\ref{prop:isometry}, $T(A) \cong S(A^q)$ is isometrically affinely homeomorphic to $T(B)\cong S(B^q)$, and so $A\cong B$ if and only if
\[
(K_0(A),K_0(A)_+,[1_A],K_1(A)) \cong (K_0(B),K_0(B)_+,[1_B],K_1(B)).
\]
Moreover, the isomorphism $\varphi \colon A \to B$ induces the given isometric map $T(B) \to T(A)$, so as in Proposition~\ref{prop:isometry}\eqref{it:isometry1}, $\varphi$ is an isometry (that is, $M \circ \varphi = L$). In other words, the $K$-localised isometric isomorphism problem \emph{does} have a positive solution in our primary setting of interest. This is the content of Corollary~\ref{cor:quig}.

Let us now turn to the definition and implementation of the quantum intertwining gap, which is inspired by the metrically equivalent version of $\dist_{\mathrm{GH}}$ presented in \cite[Section 1.3]{Rong:2010aa}.

\begin{definition} \label{def:quig}
For compact, convex metric spaces $X,Y$ and $\varepsilon>0$, define $\isom(X,Y)$ to be the set of continuous affine maps $f\colon X \to Y$ that are
\begin{itemize}
\item \emph{$\varepsilon$-isometric} ($|d_Y(f(x_1),f(x_2)) - d_X(x_1,x_2)| < \varepsilon$ for every $x_1,x_2\in X$) and
\item \emph{$\varepsilon$-invertible} (there exists a continuous affine map $g \colon Y \to X$ such that $d_Y(f \circ g(y),y) < \varepsilon$ for every $y\in Y$).
\end{itemize}
The \emph{intertwining gap} between $(X,d_X)$ and $(Y,d_Y)$ is
\begin{equation} \label{eqn:spacing}
\gamma(X,Y) = \inf\{\varepsilon>0 \mid \isom(X,Y) \times \isom(Y,X) \ne \emptyset\}.
\end{equation}
The \emph{quantum intertwining gap} between $\cs$-algebraic quantum metric Choquet simplices $(A,L)$ and $(B,M)$ is
\begin{equation} \label{eqn:atomic}
\gamma_q((A,L),(B,M)) = \gamma((T(A),\rho_L),(T(B),\rho_M)).
\end{equation}
\end{definition}

\noindent Taking a cue from \cite[Definition  5.6]{Rieffel:2010ab}, we have avoided the word `distance' because of potential failure of the triangle inequality. This shortcoming notwithstanding, in the Bauer setting the intertwining gap is at least a \emph{quasi}metric in the sense of \cite[Chapter 14]{Heinonen:2001aa} (studied also, for example, in \cite{Xia:2009aa}), namely, there is a constant $C$ $(=2)$ such that
\[
\gamma(X,Z) \le C(\gamma(X,Y) + \gamma(Y,Z)) \:\text{ for all } X,Y,Z.
\]
Again restricted to the Bauer setting, the intertwining gap also provides another description of the $\dq$-topology. In the following, by a \emph{metric Bauer simplex} we mean a Bauer simplex $X$ whose metric has been induced from a metric on $\partial_eX$ in the manner described in Section~\ref{subsection:qmbs}, that is, we identify $X$ with the probability space $\prob(\partial_eX)$ equipped with the Wasserstein metric \eqref{eqn:wass}. Note that, by Kadison's representation theorem \cite[Theorem \rm{II}.1.8]{Alfsen:1971hl} and the duality provided by Proposition~\ref{prop:isometry}\eqref{it:isometry1}, the functors $\aff$ and $S(\cdot)$ provide an equivalence between the category of metric Bauer simplices (with isometric affine maps as morphisms) and the opposite of the category of commutative $\cs$-algebraic quantum metric Bauer simplices (whose morphisms are isometries in the sense described in Definition~\ref{def:qgh}).

\begin{theorem} \label{thm:quig}
The intertwining gap $\gamma$ is a complete quasimetric on the set of affine isometric isomorphism classes of metric Bauer simplices. Moreover, in this setting $\gamma=\gamma_q$ is topologically equivalent to $\dq$, that is, $\gamma$-convergence of a sequence $(X_n)$ to $X$ is the same as $\dq$-convergence of $(\aff(X_n))$ to $\aff(X)$.
\end{theorem}

\begin{proof}
The intertwining gap $\gamma$ is symmetric by its definition. It is also apparent that $\gamma(X,Y) = 0$ if $X$ and $Y$ are isometrically affinely isomorphic. The converse implication, namely that isometric affine isomorphism of $X$ and $Y$ is necessitated by $\gamma(X,Y) = 0$, will follow from Proposition~\ref{prop:isometry}\eqref{it:isometry2} and the equivalence of $\gamma_q$ and $\dq$ demonstrated below. But let us first investigate the triangle inequality. Suppose that $X,Y,Z$ are compact, convex metric spaces with $\gamma(X,Y) = s$ and $\gamma(Y,Z) = t$. Let $\varepsilon>0$, and let $f_1 \colon X \to Y$ be an $(s+\varepsilon)$-isometry with $(s+\varepsilon)$-inverse $g_1 \colon Y \to X$ and $f_2 \colon Y \to Z$ a $(t+\varepsilon)$-isometry with $(t+\varepsilon)$-inverse $g_2 \colon Z \to Y$. Let $f=f_2\circ f_1 \colon X \to Z$ and $g= g_1 \circ g_2 \colon Z \to Y$. Then $f$ and $g$ are continuous and affine, $f$ is an $(s+t+2\varepsilon)$-isometry and for $z\in Z$ we have
\begin{align*}
d_Z(f\circ g (z),z) &= d_Z(f_2 \circ f_1 \circ g_1 \circ g_2(z),z)\\
&\le d_Z(f_2 \circ f_1 \circ g_1 \circ g_2(z),f_2 \circ g_2(z)) + d_Z(f_2 \circ g_2(z),z) \\
&\le d_Y(f_1 \circ g_1 \circ g_2(z),g_2(z)) + 2t + 2 \varepsilon\\
&\le s + 2t + 3\varepsilon.
\end{align*}
By symmetry, there similarly exists an $(s+t+2\varepsilon)$-isometry $f' \colon Z \to X$ with $2s+t+3\varepsilon$-inverse $g'$. It follows that
\[
\gamma(X,Z) \le 2s + 2t = 2\left(\gamma(X,Y)+\gamma(Y,Z)\right),
\]
so $\gamma$ is a quasimetric.

Next, let us explore the relationship between $\gamma$ and $\dq$ in the Bauer category. Let $X$ be a metric Bauer simplex. First, we claim that if $\gamma(X,Y)$ is sufficiently small, then an almost-surjective, almost-isometry $f \colon X \to Y$ witnessing this closeness cannot map $\partial_eX$ too far away from $\partial_eY$. More precisely, let $\varepsilon>0$ and set
\[
\delta := \frac{1}{2}\inf\left\{d_X\left(x,\frac{1}{2}(x_1+x_2)\right) \mid x\in\partial_eX,\, x_i\in X,\, d_X(x_i,x) \ge \varepsilon\:, i=1,2\right\} \in \left(0,\frac{\varepsilon}{2}\right]
\]
(assuming that $\varepsilon$ is small enough for the defining set to be nonempty). Suppose that $Y$ is a metric Bauer simplex with $\gamma(X,Y) < \delta$. Let $f \colon X \to Y$ be a $\delta$-isometry with $\delta$-inverse $g$. Then, we must have
\[
f(x) \in (\partial_eY)_{\frac{1}{2}(\varepsilon + 3\delta)} \subseteq (\partial_eY)_{2\varepsilon} \: \text{ for every } x\in \partial_eX.
\]
To see this, suppose that $x\in \partial_eX$ is such that $B_{\frac{1}{2}(\varepsilon + 3\delta)}(f(x)) \cap \partial_eY = \emptyset$. Let $y_1,y_2 \in Y$ be two antipodal points on the $\frac{1}{2}(\varepsilon+3\delta)$-sphere centred at $f(x)$, so that $f(x) = \frac{1}{2}(y_1+y_2)$ and $d_Y(y_1,y_2) = \varepsilon+3\delta$. Let $x_1,x_2\in X$ such that $d_Y(f(x_i),y_i)<\delta$ for $i=1,2$ (namely, $x_i=g(y_i)$) and set $x' = \frac{1}{2}(x_1+x_2)$. Then,
\[
d_X(x_1,x_2) > d_Y(f(x_1),f(x_2)) - \delta > d_Y(y_1,y_2) - 3\delta = \varepsilon.
\]
On the other hand, we have by convexity \eqref{eqn:convex} of the metric $d_Y$ that
\begin{align*}
d_Y(f(x),f(x')) &= d_Y\left(\frac{1}{2}(y_1+y_2),\frac{1}{2}(f(x_1)+f(x_2))\right)\\
&\le \frac{1}{2}\left(d_Y(y_1,f(x_1)) + d_Y(y_2,f(x_2))\right)\\
&< \delta,
\end{align*}
so
\[
d_X\left(x,\frac{1}{2}(x_1+x_2)\right) = d_X(x,x') < d_Y(f(x),f(x')) + \delta < 2\delta,
\]
which contradicts the definition of $\delta$. So indeed, $f(x) \in (\partial_eY)_{\frac{1}{2}(\varepsilon + 3\delta)} $ as claimed. With this bookkeeping taken care of (and still assuming that $\gamma(X,Y)<\delta$), we can now follow the proof of \cite[Lemma 1.3.3]{Rong:2010aa} (suitably tailored to the affine setting) to equip $X \sqcup Y$ with metric Bauer structure. Namely, we define an admissible metric on $\partial_eX \sqcup \partial_eY$ (that is, one that agrees with $d_X$ and $d_Y$ on $X$ and $Y$, respectively) by
\begin{equation} \label{eqn:iii}
d(x,y) := \inf_{z\in X}(d_X(x,z) + d_Y(f(z),y)) + \frac{\delta}{2}
\end{equation}
for $x\in\partial_eX$ and $y\in\partial_eY$, and then extend $d$ to $\prob(\partial_eX \sqcup \partial_eY)$ via \eqref{eqn:wass}. (Any skeptic of the triangle inequality is invited to inspect the proof of \cite[Lemma 1.3.3]{Rong:2010aa}, where this is demonstrated.) Let $x\in \partial_eX$ and let $y \in \partial_eY$ with $d_Y(f(x),y) < \frac{1}{2}(\varepsilon + 3\delta)$. Then by definition,
\begin{equation} \label{eqn:rong}
d(x,y) \le  d_X(x,x) + d_Y(f(x),y) + \frac{\delta}{2} < \frac{\varepsilon}{2} + 2\delta \le \frac{3\varepsilon}{2},
\end{equation}
so $\partial_eX \subseteq (\partial_eY)_{\frac{3\varepsilon}{2}}$, and therefore $X \subseteq Y_{\frac{3\varepsilon}{2}}$ by convexity \eqref{eqn:convex} and the Krein--Milman theorem (whose subsequent use is sometimes stealthy). Note that, by \eqref{eqn:rong}, we have
\[
d(x,f(x)) < \frac{\varepsilon}{2} + \frac{3\delta}{2} + \frac{3\varepsilon}{2}
\]
for every $x\in\partial_eX$, and hence also for every $x\in X$. Therefore, for $y\in Y$,
\[
d(g(y),y) \le d(g(y),f(g(y))) + d(f(g(y)),y)
< \frac{\varepsilon}{2} + \frac{3\delta}{2} + \frac{3\varepsilon}{2} + \delta \le \frac{13\varepsilon}{4},
\]
so $Y \subseteq (X)_{\frac{13\varepsilon}{4}}$. It follows from Proposition~\ref{prop:gh} (with $(E,L) = (\aff(\prob(X \sqcup Y)),L_{W_d})$) that
\[
\dq(\aff(X),\aff(Y)) < \frac{13\varepsilon}{4} < 4\varepsilon.
\]
As mentioned above, a consequence of this observation is the positive definiteness of $\gamma$, so we have also shown that $\gamma$ is a quasimetric.

Suppose conversely that $\dq(\aff(X),\aff(Y)) < \varepsilon$. As in Proposition~\ref{prop:gh}, there are affine isometric embeddings of $X$ and $Y$ into the state space $Z$ (in fact, the closed convex hull of $X \sqcup Y$) of a quantum metric Choquet simplex such that $X\subseteq Y_\varepsilon$ and $Y \subseteq X_\varepsilon$. To find suitable maps $f\colon X \to Y$ and $g\colon Y \to X$, we apply the Yannelis--Prabhakar selection theorem \cite[Theorem 3.1]{Yannelis:1983aa} to the compact metric space $\partial_eX$, the convex space $Y$ and the subsets $S(x) := \{y \in Y \mid d_Z(x,y) < \varepsilon\}$ for $x\in\partial_eX$. Each $S(x)$ is nonempty and convex and, for every $y\in Y$, $S^{-1}(y):=\{x\in \partial_eX \mid y \in S(x)\} = \partial_eX \cap B_\varepsilon(y)$ is an open subset of $\partial_eX$. Under these circumstances, the Yannelis--Prabhakar theorem provides a continuous function $f\colon \partial_eX \to Y$ such that, for every $x\in \partial_eX$, $f(x) \in S(x)$, that is $d_Z(x,f(x))<\varepsilon$. Similarly, there is a continuous function $g\colon \partial_eY \to X$ such that, for every $y\in \partial_eY$, $d_Z(y,g(y))<\varepsilon$. Extend $f$ and $g$ to continuous affine functions $X \to Y$ and $Y \to X$, respectively. Notice that, for $x_1,x_2 \in \partial_eX$, we have
\begin{align*}
-2\varepsilon < - d_Z(x_1,f(x_1)) - d_Z(x_2,f(x_2) &\le d_Y(f(x_1),f(x_2)) - d_X(x_1,x_2)\\
&\le d_Z(f(x_1),x_1) + d_Z(x_2,f(x_2))\\
&< 2\varepsilon.
\end{align*}
In other words,
\[
|d_Y(f(x_1),f(x_2)) - d_X(x_1,x_2)| <2\varepsilon.
\]
By convexity \eqref{eqn:convex}, this remains true for $x_1,x_2\in X$, because for $x=\frac{1}{n}\sum_{i=1}^ne_i$ with $e_i\in\partial_eX$, we have
\begin{align*}
d_Z(x,f(x)) = d_Z\left(\frac{1}{n}\sum_{i=1}^ne_i,\frac{1}{n}\sum_{i=1}^nf(e_i)\right) &\le \frac{1}{n}\sum_{i=1}^nd_Z(e_i,f(e_i))\\
&\le \sup_{x\in X} d_Z(x,f(x))\\
&< \varepsilon,
\end{align*}
so by Krein--Milman, $d_Z(x,f(x)) < \varepsilon$ for every $x\in X$. It follows that $f \colon X \to Y$ is a $2\varepsilon$-isometry. Similarly, so is $g$. For $y\in \partial_eY$, we also have
\[
d_Y(y,f\circ g(y)) \le d_Z(y,g(y)) + d_Z(g(y),f(g(y))) < 2\varepsilon,
\]
so $g$ is a $2\varepsilon$-inverse of $f$ (and similarly the other way round). In other words, we have shown that
\[
\gamma(X,Y)  = \gamma_q(\aff(X),\aff(Y)) \le 2\dq(\aff(X),\aff(Y)),
\]
completing the proof of topological equivalence of $\gamma$ and $\dq$.

Finally, let us show that $\gamma$ is complete. Suppose that $(X_n,d_n)_{n\in\nn}$ is a $\gamma$-Cauchy sequence of metric Bauer simplices. Passing to a subsequence if necessary, we may assume that, for every $m\ge n\in\nn$, $\gamma(X_n,X_m)$ is sufficiently small so that there exists an admissible metric on $\partial_eX_n \sqcup \partial_eX_m$ defined as in \eqref{eqn:iii}, with the property that the Hausdorff distance between $\partial_eX_n$ and $\partial_eX_m$ is less than $2^{-n}$. It follows that $(\partial_eX_n,d_n)$ is a Cauchy sequence in the metric space of isometry classes of compact metric spaces equipped with the Gromov--Hausdorff distance $\dist_{\mathrm{GH}}$ defined in \eqref{eqn:dgh}. This metric space is complete (a fact whose proof can be found in \cite[Section 1]{Rong:2010aa}), so there exists a compact metric space $(\widehat{X},\widehat{d})$ such that $\lim_{n\to\infty}\dist_{\mathrm{GH}}(\partial_eX_n,\widehat{X})=0$. Let $(X,d)=(\prob(\widehat{X}),W_{\widehat{d}})$ be the corresponding metric Bauer simplex. By \cite[Proposition 4.7]{Rieffel:2004aa}, $\dq(\aff(X_n),\aff(X)) \le \dist_{\mathrm{GH}}(\partial_eX_n,\partial_eX_)$, so $\lim_{n\to\infty}\dq(\aff(X_n),\aff(X))=0$ as well. Since we have just seen that $\gamma_q$ and $\dq$ are topologically equivalent, we conclude that $X$ is the $\gamma$-limit of the Cauchy sequence $(X_n)$, and hence that $\gamma$ is a complete. 
\end{proof}

\begin{remark} \label{rem:quig}
Alternatively, and more closely following \cite[Definition 1.3.2]{Rong:2010aa}, we could instead consider the distance $\widehat{\dist}_{GH}(X,Y)$ defined to be the infimum over all $\varepsilon>0$ for which there exists a (not necessarily continuous) affine map $f\colon X\to Y$ that is $\varepsilon$-isometric (just as in Definition~\ref{def:quig}) and $\varepsilon$-surjective (that is, $f(X)_\varepsilon=Y$). Let us call this version of $\dist_{GH}$ the \emph{Fukaya--Gromov--Hausdorff distance} or simply the \emph{Fukaya distance} (its nonaffine version appearing in \cite[Chapter \rm{I}]{Fukaya:1990aa}). Then we define the \emph{quantum Fukaya distance} between $\cs$-algebraic quantum metric Choquet simplices $(A,L)$ and $(B,M)$ to be
\[
\widehat{\dq}((A,L),(B,M)) = \widehat{\dist}_{GH}((T(A),\rho_L),(T(B),\rho_M)).
\]
This is strong enough to have used in the proof of Theorem~\ref{thm:quig}, and the conclusion would be the same (including the fact that $\widehat{\dq}$ is a quasimetric on metric Bauer simplices). So, on the set of affine isometric isomorphism classes of commutative $\cs$-algebraic quantum metric Bauer simplices, we have
\[
\widehat{\dq} \le \gamma_q  \le 2\dq 
\]
and the $\dq$-, $\widehat{\dq}$- and $\gamma_q$-topologies all coincide.
\end{remark}

An immediate consequence of Proposition~\ref{prop:isometry}\eqref{it:isometry1}, Theorem~\ref{thm:quig} and \cite[Theorem 13.5]{Rieffel:2004aa} (for the description of total boundedness) is our partial answer to Question~\ref{q:qmcs}.

\begin{corollary} \label{cor:complete}
$(\mathbb{QMBS},\dq)$ is complete, and any subset of $\mathbb{QMBS}$ whose elements have uniformly bounded radii and covering growths in the sense of \eqref{eqn:covering} is relatively compact.
\end{corollary}

As Li remarks in \cite[Section 4]{Li:2006aa}, it might be more consistent with the principles of noncommutative geometry to consider quantum distances that are defined at the level of \emph{function} spaces rather than \emph{state} spaces. This intuition led to the \emph{order-unit quantum Gromov--Hausdorff distance} \cite[Definition 4.2]{Li:2006aa} and its variants \cite[Definition 3.3]{Li:2003aa}, \cite[Definition 3.1]{Kerr:2009aa} designed to distinguish $\cs$-algebraic (not just order-theoretic) structure as discussed before Definition~\ref{def:quig}. By classification, $\gamma_q$ also admits a description at the function space level in the $K$-connected, classifiable $\cs$-QMBS category, the key point in this case being that tracial morphisms can be lifted to $\cs$-algebraic ones.

\begin{corollary} \label{cor:quig}
Let $\mathcal{K}$ denote the set of isometric isomorphism classes of $\cs$-algebraic quantum metric Bauer simplices $(A,L)$, where $A$ is a unital, $K$-connected, classifiable $\cs$-algebra and $(A,L)$ is equivalent to $(B,M)$ if there is a $^*$-isomorphism $\varphi \colon A \to B$ such that $M \circ \varphi = L$. Given a countable abelian group $G_1$ and a countable, simple, weakly unperforated ordered abelian group $(G_0,G_0^+)$ with distinguished order unit $g_0\in G_0^+$, write $G= (G_0,G_0^+,g_0,G_1)$ and
\[
\mathcal{K}_{G} = \{[(A,L)] \in \mathcal{K} \mid (K_0(A),K_0(A)_+,[1_A],K_1(A)) \cong G\}.
\]
Then, for every such $G$, $\gamma_q$ and $\widehat{\dq}$ are quasimetrics on $\mathcal{K}_{G}$ that are topologically equivalent  to $\dq$. Moreover, $\gamma_q$ admits the description
\begin{align*}
\gamma_q((A,L),(B,M)) = \inf&\{\varepsilon>0 \mid \text{ there exist } \varphi \in \emb(A,B), \psi \in \emb(B,A) \text{ with}\\
&T(\varphi) \in \isom(T(B),T(A)), T(\psi) \in \isom(T(A),T(B))\},
\end{align*}
where $\isom(T(B),T(A))$ and $\isom(T(A),T(B))$ are as in Definition~\ref{def:quig} relative to the metrics $\rho_L$ on $T(A)$ and $\rho_M$ on $T(B)$.
\end{corollary}

\begin{proof}
The only thing left to observe is that, if $[(A,L)],[(B,M)] \in \mathcal{K}_G$ are such that $\widehat{\dq}((A,L),(B,M))=0$, then there is a $^*$-isomorphism $\varphi \colon A \to B$ such that $M \circ \varphi = L$. (Everything else has been covered by Remark~\ref{rem:quig}, Theorem~\ref{thm:quig} and the discussion preceding Definition~\ref{def:quig}.) But this follows from classification (which also provides the stated description of $\gamma_q$). Since $\widehat{\dist}_{GH}((T(A),\rho_L),(T(B),\rho_M)) = 0$, there is an affine isometric isomorphism $\alpha\colon T(B) \to T(A)$. Since $A$ and $B$ are both in $\mathcal{K}_G$, there is also an isomorphism of $K$-theory that by $K$-connectedness is compatible with $\alpha$. We can package these into a morphism of the `total invariant' $\underline{K}T_u$ (in the notation of \cite{Carrion:wz}, and appealing to \cite[Theorem 3.9]{Carrion:wz}) and then lift to the required $^*$-isomorphism $\varphi \colon A \to B$ (by \cite[Theorem B]{Carrion:wz}). Moreover, we have $M \circ \varphi = L$ by Proposition~\ref{prop:isometry}\eqref{it:isometry1} because $T(\varphi)$ is isometric.
\end{proof}

\vspace{2mm}

\section{Quantum systems} \label{section:qcp}

\subsection{$\cs$-algebras arising as tracial quantum crossed products} \label{subsection:range}

In \cite{Jacelon:2022wr}, model $\cs$-dynamical systems $(A,\alpha)$, with $A$ a classifiable $\cs$-algebra, were built to witness either a desired statistical feature of the automorphism (or tracially nondegenerate endomorphism) $\alpha$ or a prescribed isomorphism class of the simplex $T(A\rtimes_\alpha\zz)$. The technique in both cases was the same: start with a favourable topological dynamical system $(X,h)$, construct a $K$-connected, classifiable $\cs$-algebra $A$ with $\partial_e(T(A))\cong X$ (as in Example~\ref{ex:cqmcs}.\ref{it:ex2}), then use classification to lift $h$ to $\alpha\in\End(A)$ (an automorphism if $h$ is a homeomorphism). The lift $\alpha$ can always be chosen to have \emph{finite Rokhlin dimension}, which guarantees that $A\rtimes_{\alpha_h}\nn$ is itself classifiable with $T(A\rtimes_{\alpha_h}\nn) \cong \prob(X)^h$, the simplex of $h$-invariant Borel probability measures on $X$ (see \cite[Lemma 4.8]{Jacelon:2022wr}), and is in fact generic among nondegenerate endomorphisms of a separable, $\js$-stable $\cs$-algebra (see \cite{Hirshberg:2015wh,Szabo:2019te} and also \cite[Definition 4.6, Lemma 4.7]{Jacelon:2022wr}). We consider this procedure sufficiently fundamental to merit its cementation as follows.

\begin{definition} \label{def:qcp}
Let $X$ be a compact metrisable space and $h\colon X\to X$ continuous (often a homeomorphism but not necessarily). A \emph{tracial quantum system associated with $(X,h)$} consists of a $\cs$-algebra $A_X$ with $\partial_e(T(A_X))$ homeomorphic to $X$, together with a tracially nondegenerate endomorphism $\alpha_h$ of $A_X$ such that $T(\alpha_h)|_{\partial_e(T(A_X))}=h$ under the homeomorphism $\partial_e(T(A_X))\cong X$. The corresponding \emph{tracial quantum crossed product associated with $(X,h)$} is the $\cs$-algebra $A_X\rtimes_{\alpha_h}\nn$. If $\alpha$ is an automorphism, we may for emphasis call the system \emph{invertible}. If $A_X$ is $K$-connected and classifiable and has continuous scale (see Remark~\ref{rem:cqmcs}) and $\alpha_h$ has finite Rokhlin dimension, then we call $(A_X,\alpha_h)$ a \emph{Stein system associated with $(X,h)$}.
\end{definition}

\begin{remark} \label{rem:endo}
As in \cite[Section 4.4]{Jacelon:2022wr}, $A\rtimes_{\alpha}\nn$ denotes the (reduced) crossed product of a $\cs$-algebra $A$ by an endomorphism $\alpha$ in the sense of \cite{Cuntz:1982uk}. Namely, it is the cutdown of $\underset{\to}{A}\rtimes_{\underset{\to}{\alpha}}\zz$ by $1_{M(A)}$, where $\underset{\to}{\alpha}$ is the extension of $\alpha$ to an automorphism of
\[
\underset{\to}{A} := \varinjlim\left(\begin{tikzcd}[column sep=small, row sep=small]A \arrow[r,"\alpha"] & A \arrow[r,"\alpha"] & A \arrow[r,"\alpha"] & \ldots\end{tikzcd}\right)
\]
and $A$ is mapped to $\underset{\to}{A}$ as the first component of the inductive limit (injectively if $\alpha$ is injective, which for us is always the case). As explained in \cite[Section 3]{Stacey:1993uc}, $A\rtimes_{\alpha}\nn$ can equivalently be described as a $\cs$-algebra in which the action on $A$ is implemented by an isometry, and which is universal for covariant representations. In this picture, $A\rtimes_{\alpha}\nn$ is densely spanned by elements of the form $at^mt^{*n}$, where $a\in A$, $m,n\in\nn_0$ and $t$ is an isometry such that $tbt^*=\alpha(b)$ for every $b\in A$. If $A$ and $\alpha$ are unital, then $A\rtimes_{\alpha}\nn = \underset{\to}{A}\rtimes_{\underset{\to}{\alpha}}\zz$. If $\alpha$ is invertible, then $A\rtimes_{\alpha}\nn \cong A\rtimes_{\alpha}\zz$. If $\alpha$ has finite Rokhlin dimension, then (by design) so does $\underset{\to}{\alpha}$. Allowing for endomorphisms means that we can capture a wider range of examples to demonstrate results such as Proposition~\ref{prop:breiman} and Theorem~\ref{thm:metricdeform}.\eqref{it:wave2}, but it is also sometimes a necessary part of our analysis. Indeed, the classification theory to which we appeal in Theorem~\ref{thm:cf} cannot guarantee us a continuous field of invertible tracial quantum systems, even if, as in Section~\ref{subsection:inherit}, the topological systems are invertible themselves.
\end{remark}

\begin{theorem} \label{thm:range}
Every $\js_0$-stable classifiable $\cs$-algebra is stably isomorphic to a tracial quantum crossed product associated with a minimal homeomorphism of a zero-dimensional space.
\end{theorem}

\begin{proof}
The proof is the same as that of \cite[Theorem 5.1]{Jacelon:2022wr}, except that the observable space $A$ is chosen not to be $KK$-contractible, but rather a suitable $\js_0$-stable algebra. To wit, let $B$ be a $\js_0$-stable classifiable $\cs$-algebra (the $\cs$-algebra whose isomorphism class we wish to attain). We may assume up to stable isomorphism that $B$ has continuous scale. As in the proof of \cite[Proposition 3.5]{Rordam:1995fr}, let $H$ be a countable abelian group that admits an automorphism $\kappa\colon H\to H$ with $\ker(\id-\kappa)\cong K_1(B)$ and $\coker(\id-\kappa)\cong K_0(B)$. (Note that the assumption in \cite[Proposition 3.5]{Rordam:1995fr} that $G_1$ be torsion free is to ensure that $H$ is torsion free, but that is not a concern for us here.)  By \cite[Theorem 7.11]{Gong:2020vg}, there is a $\js_0$-stable classifiable $\cs$-algebra $A$ with continuous scale such that $K_0(A)\cong H$ and $K_1(A)=0$, and (by \cite[Theorem 3]{Downarowicz:1991te} and \cite[Theorem 13.1]{Gong:2020vg}) such that $A$ serves as the observable space of a minimal, zero dimensional system $(X,h)$ whose simplex of invariant measures is affinely homeomorphic to $T(B)$. We can moreover choose the automorphism $\alpha=\alpha_h\colon A\to A$ that lifts $h$ to have finite Rokhlin dimension (see \cite[Lemma 4.7]{Jacelon:2022wr}), so that the crossed product $A\rtimes_\alpha\zz$ has continuous scale (see the proof of \cite[Theorem 5.1]{Jacelon:2022wr}) and is classifiable with $T(A\rtimes_\alpha\zz)\cong T(B)$ (see \cite[Lemma 4.8]{Jacelon:2022wr}). The Pimsner--Voiculescu sequence
\begin{equation} \label{eqn:pv}
\begin{tikzcd}
K_0(A) \arrow[r,"1-\alpha_*"] & K_0(A) \arrow[r,"\iota_*"] & K_0(A\rtimes_\alpha\zz)\arrow[d]\\
K_1(A\rtimes_\alpha\zz) \arrow[u] & K_1(A) \arrow[l,"\iota_*"] & K_1(A) \arrow[l,"1-\alpha_*"]
\end{tikzcd}
\end{equation}
(see \cite[Theorem 2.4]{Pimsner:1980yu} or \cite[Theorem 10.2.1]{Blackadar:1998qf}) gives $K_0(A\rtimes_\alpha\zz)\cong K_0(B)$ and $K_1(A\rtimes_\alpha\zz)\cong K_1(B)$. Since the map $\iota_*\colon K_0(A) \to A\rtimes_\alpha\zz$ is surjective and every trace on $A\rtimes_\alpha\zz$ restricts to one on $A$, we also conclude that $A\rtimes_\alpha\zz$ is $K$-connected. It then follows from \cite[Theorem 15.6]{Gong:2020vg} that $A\rtimes_\alpha\zz\cong B$.
\end{proof}

Theorem~\ref{thm:range} provides an affirmative answer to the stably projectionless, $K$-connected case of the following pressing question.

\begin{question} \label{q:range}
Is every stably finite classifiable $\cs$-algebra stably isomorphic to a crossed product of a classifiable $\cs$-algebras by an action of the integers?
\end{question}

\begin{remark} \label{rem:groups}
Definition~\ref{def:qcp} makes sense not just for single endomorphisms $h$ but more generally for group actions $G\to\Aut(C(X))$. But whereas the existence of a noncommutative observable space $A$ and lifted action $\alpha_h$ is guaranteed for integer actions, this is not so clear in the general setting. A relevant question that should be addressed is, when can a group action on the total invariant $\underline{K}T_u$ of a classifiable $\cs$-algebra be lifted to an action on the algebra itself? This question has a positive answer if $G$ is a free group (a fact which is used in \cite{Gardella:2022ab} and Theorem~\ref{thm:free}) or if $G$ is finite and $A$ absorbs the UHF algebra $M_{|G|^\infty}$ (see \cite[Corollary 9.13]{Carrion:wz}, which generalises \cite[Corollary 2.13]{Barlak:2017aa}, and also Example~\ref{ex:sphere}). At least in these cases, then, the notion of a tracial quantum crossed product is a meaningful one.
\end{remark}

Because free groups $\ff_n$ are not amenable for $n\ge2$, tracial quantum crossed products $A_X\rtimes_{\alpha_h,r}\ff_n$ associated with actions $h \colon \ff_n \to \Aut(C(X))$ will not necessarily be nuclear. But if $(A_X,\alpha_h)$ is a Stein system (meaning that $A_X$ has continuous scale, is $K$-connected and classifiable with $\partial_eT(A_X) \cong X$, and $\alpha_h \colon \ff_n \to \Aut(A_X)$ is an action with finite Rokhlin dimension such that $T(\alpha_h)|X=h$), then there is still much that can be said about the structure of the reduced crossed product $A_X\rtimes_{\alpha_h,r}\ff_n$ (see Theorem~\ref{thm:free} below). Note moreover that, by the results of \cite[Section 11]{Szabo:2019te}, we can in fact find an $\alpha_h$ with Rokhlin dimension $\le1$ (a property that is shown in \cite[Theorem 11.14]{Szabo:2019te} to be generic for actions of finitely generated, residually finite groups on separable, $\js$-stable $\cs$-algebras). As in the proof of Theorem~\ref{thm:range} (and discussed further in Corollary~\ref{cor:dynqmcs}), the results of \cite{Elliott:2020wc,Gong:2020vg} also give us the option of choosing $A_X$ to be a $\js_0$-stable or $\mathcal{W}$-stable stably projectionless classifiable $\cs$-algebra (with $\mathcal{W}$-stability being equivalent to $KK$-contractibility, that is, trivial $K$-theory).

\begin{theorem} \label{thm:free}
Suppose that $h \colon \ff_n \to \Aut(C(X))$ is a uniquely ergodic action of the free group $\ff_n=\langle g_1,\dots,g_n \rangle$, $n\ge 2$, on a compact metrisable space $X$. Let $(\mathcal{W}_X,\alpha_h)$ be a $\mathcal{W}$-stable Stein system associated with $(X,h)$. Then, the reduced crossed product $\mathcal{W}_X\rtimes_{\alpha_h,r}\ff_n$ is a simple, separable, exact, nonnuclear, monotracial, $KK$-contractible $\cs$-algebra that satisfies the UCT.
\end{theorem}

\begin{proof}
Separability of $\mathcal{W}_X$ implies that $\mathcal{W}_X\rtimes_{\alpha_h,r}\ff_n$ is also separable. The UCT holds by the results of \cite{Higson:2001aa,Meyer:2006aa}, because $\ff_n$ has the Haagerup property \cite{Haagerup:1978aa}. The Pimsner--Voiculescu sequence for $\ff_n$-actions (see \cite[Theorem 3.5]{Pimsner:1982aa} or \cite[Theorem 10.8.1]{Blackadar:1998qf}) implies that $K_*\left(\mathcal{W}_X\rtimes_{\alpha_h,r}\ff_n\right)=0$, which implies that $\mathcal{W}_X\rtimes_{\alpha_h,r}\ff_n$ is $KK$-contractible. Simplicity is a consequence of Kishimoto's theorem \cite[Theorem 3.1]{Kishimoto:1981tg} because finite Rokhlin dimension implies strong outerness (by \cite[Theorem 7.8 (3)$\Rightarrow$(1)]{Gardella:2021tb}, which can be adapted to the nonunital setting as in \cite[Lemma 4.8]{Jacelon:2022wr}).

Let us show that $\mathcal{W}_X\rtimes_{\alpha_h,r}\ff_n$ has a unique trace (which is necessarily bounded because all traces on $\mathcal{W}_X$ are bounded and any approximate unit for $\mathcal{W}_X$ passes to an approximate unit for $\mathcal{W}_X\rtimes_{\alpha_h,r}\ff_n$). Since $h$ is uniquely ergodic, there is a unique $\alpha_h$-invariant trace $\tau_\mu$ corresponding to the unique $h$-invariant measure $\mu$, which extends to a trace on $\mathcal{W}_X\rtimes_{\alpha_h,r}\ff_n$ by composition with the canonical conditional expectation onto $\mathcal{W}_X$. Following \cite[Proposition 2.3]{Liao:2016ui}, we use finite Rokhlin dimension to show that this extension is unique. As in \cite{Liao:2016ui}, it suffices to prove that $\tau(au_g)=0$ for every $\tau\in T(\mathcal{W}_X\rtimes_{\alpha_h,r}\ff_n)$, every $a$ in the unit ball of $\mathcal{W}_X$ and every $g \in \setminus \{1_{\ff_n}\}$, where $u_g$ is the unitary implementing the action of $\alpha_h(g)$. We would then know that $\tau$ agrees with the canonical extension of $\tau_\mu$ on finite sums, that is, $\tau(\sum_{g\in\ff_n}a_gu_g)=\tau_\mu(a_{1_{\ff_n}})$, and hence that the canonical extension is unique. Fix such a $(\tau,a,g)$. Let $\varepsilon>0$, let $H\le\ff_n$ be a normal subgroup of finite index such that $g\notin H$, and let $\delta=\frac{\varepsilon}{2|\ff_n/H|+1}$. We may assume that the Rokhlin dimension of $\alpha_h$ is at most $1$. By \cite[Definition 5.8 and Proposition 5.5]{Szabo:2019te}, we can find positive contractions $(f^{(l)}_{\overline k})_{\overline k \in \ff_n/H}^{l\in\{0,1\}}$ in $\mathcal{W}_X$ such that
\begin{itemize}
\item $\left\|\left(\sum_{l=0}^1\sum_{\overline k \in \ff_n/H}f^{(l)}_{\overline k}\right)a-a\right\| < \delta$, and
\item $\left\|\left(f^{(l)}_{\overline k}\right)^{1/2}a\alpha_h(g)\left(\left(f^{(l)}_{\overline k}\right)^{1/2}\right)\right\| < \delta$ for $l\in\{0,1\}$ and $\overline k\in\ff_n/H$.
\end{itemize}
Then,
\begin{align*}
\tau(au_g) &\le \left| \sum_{l=0}^1\sum_{\overline k \in \ff_n/H}\tau\left(au_gf^{(l)}_{\overline k}\right) \right| + \delta\\
&\le \sum_{l=0}^1\sum_{\overline k \in \ff_n/H} \left| \tau\left(\left(f^{(l)}_{\overline k}\right)^{1/2}au_g\left(f^{(l)}_{\overline k}\right)^{1/2}\right) \right| + \delta\\
&= \sum_{l=0}^1\sum_{\overline k \in \ff_n/H} \left| \tau\left(\left(f^{(l)}_{\overline k}\right)^{1/2}a\alpha_h(g)\left(\left(f^{(l)}_{\overline k}\right)^{1/2}\right) u_g \right) \right| + \delta\\
& < (2|\ff_n/H|+1)\delta = \varepsilon.
\end{align*}
Since $\varepsilon$ is arbitrary, it follows that $\tau(au_g)=0$ and hence that $\mathcal{W}_X\rtimes_{\alpha_h,r}\ff_n$ has a unique trace.

To show exactness, we extend $\alpha_h$ to an action on the minimal unitisation $\widetilde{\mathcal{W}_X}$ of $\mathcal{W}_X$, still denoted by $\alpha_h\colon\ff_n\to\Aut(\widetilde{\mathcal{W}_X})$. The reduced crossed product $\widetilde{\mathcal{W}_X}\rtimes_{\alpha_h,r}\ff_n$ is isomorphic to the reduced amalgamated free product $\free_{i=1}^n \left(\widetilde{\mathcal{W}_X}\rtimes_{\alpha_h(g_i)}\zz,E_i\right)$, where $E_i$ is the conditional expectation of $\widetilde{\mathcal{W}_X}\rtimes_{\alpha_h(g_i)}\zz$ onto $\widetilde{\mathcal{W}_X}$. By \cite[Theorem 3.2]{Dykema:2004aa}, it follows that $\widetilde{\mathcal{W}_X}\rtimes_{\alpha_h,r}\ff_n$ is exact, and hence so is the subalgebra $\mathcal{W}_X\rtimes_{\alpha_h,r}\ff_n$.

Finally, let us show that $\mathcal{W}_X\rtimes_{\alpha_h,r}\ff_n$ is not nuclear, or equivalently by \cite[Th\'{e}or\`{e}me 4.5]{Delaroche:1987aa}, that $\alpha_h$ is not an amenable action. To see this, we proceed as in the proof of \cite[Theorem 2.5 (2)$\Rightarrow$(6)]{Gardella:2022ab}. More precisely, suppose for a contradiction that $\alpha_h\colon\ff_n\to\Aut(\mathcal{W}_X)$ is an amenable action, meaning that the universal enveloping $\mathrm{W}^*$-dynamical system of $\alpha_h\colon\ff_n\to\Aut(\mathcal{W}_X)$ described in \cite[Section 2.6]{Ozawa:2021aa} is amenable in the sense of \cite[Theorem 2.1]{Ozawa:2021aa}. Then, so is the induced action on the minimal unitisation $\widetilde{\mathcal{W}_X}$ (because the enveloping von Neumann system is the same as for the action on $\mathcal{W}_X$). By the characterisation of amenability in terms of the quasi-central approximation property (see \cite[Theorem 3.2]{Ozawa:2021aa}), we would then in particular have the following: for any $\varepsilon>0$ and any fixed finite set $K\subseteq\ff_n$, there exists a finitely supported element $\xi\in L^2(\ff_n,\widetilde{\mathcal{W}_X})$ such that
\[
\|\xi\|\le1 \quad \text{and} \quad \|\langle\xi,\widehat\alpha_h(g)(\xi)\rangle-1\|<\varepsilon \:\text{ for every $g\in K$.}
\]
Here, $\widehat\alpha_h$ denotes the induced diagonal action $(\widehat\alpha_h(g)(\xi))(k):=\alpha_h(g)(\xi(g^{-1}k))$ on the space $L^2(\ff_n,\widetilde{\mathcal{W}_X})$, which is equipped with its usual Hilbert $\cs$-module structure. Define $\theta\colon\ff_n\to C(X)$ by
\[
\theta(g)(\tau)=\tau(\langle\xi,\widehat\alpha_h(g)(\xi)\rangle) \:\text{ for $g\in\ff_n,\tau\in X=\partial_eT(\mathcal{W}_X)\subseteq T(\widetilde{\mathcal{W}_X})$.}
\]
This function is finitely supported and is of `positive type with respect to $h$' meaning that, for every finite subset $F\subseteq\ff_n$, the $C(X)$-valued matrix $(h(g)(\theta(g^{-1}k)))_{g,k\in F}$ is positive. It also satisfies
\[
\|\theta(g)-1\| = \sup_{\tau\in\partial_eT(\mathcal{W}_X)} |\tau(\langle\xi,\widehat\alpha_h(g)(\xi)\rangle-1)| \le \|\langle\xi,\widehat\alpha_h(g)(\xi)\rangle-1\| <\varepsilon
\]
for every $g\in K$, and $\|\theta(1_{\ff_n})\|\le\|\langle \xi,\xi\rangle\|=\|\xi\|^2\le1$. By \cite[Theorem 2.4]{Gardella:2022ab} (which is a restatement of results from \cite{Delaroche:1987aa}), the existence of such a $\theta$ for arbitrary $K$ and $\varepsilon$ implies that $h \colon \ff_n \to \Aut(C(X))$ is an amenable action. Since there exists an $h$-invariant measure on $X$, it then follows from \cite[Lemma 2.2]{Gardella:2023aa} that $\ff_n$ is amenable, which is the desired contradiction.
\end{proof}

Note that every $\ff_n$ does admit a uniquely ergodic action on some $X$. This could be a trivial extension of a uniquely ergodic $\zz$-action (such as the trivial action on a singleton) or something more interesting: in \cite[Section 2]{Cortez:2008aa} (a reference for which we thank Michal Doucha), it is explained how every discrete, finitely generated, residually finite group $G$ admits an action on a profinite completion $\overrightarrow{G}$, a compact group on which the only $G$-invariant measure is the Haar measure. We might think of the $\cs$-algebras $\mathcal{W}_X\rtimes_{\alpha_h,r}\ff_n$ associated with such actions as nonamenable analogues of $\mathcal{W}$. It would certainly be interesting to investigate other structural properties of these algebras, in particular, their stable rank.

\subsection{Inheritance of QMCS structure} \label{subsection:inherit}

Recall from Remark~\ref{rem:endo} the automorphic extension $(\underset{\to}{A},\underset{\to}{\alpha})$ of an endomorphism $\alpha \in \End(A)$. Here in Section~\ref{subsection:inherit} and later in Section~\ref{subsection:dynamics}, we consider only \emph{invertible} topological dynamical systems (that is, homeomorphisms) $h \colon X \to X$, and (except in Corollary~\ref{cor:dynqmcs}, where stably projectionless models come in handy) we are primarily interested in \emph{unital} tracial quantum systems $(A_X,\alpha_h)$ associated with $(X,h)$ (that is, we require $A=A_X$ to be unital and $\alpha=\alpha_h$ to be a unital $^*$-monomorphism). In this case, we have $A \rtimes_\alpha \nn = \underset{\to}{A}\rtimes_{\underset{\to}{\alpha}}\zz$,  and the canonical inclusion map $A \hookrightarrow \underset{\to}{A}$ induces affine homeomorphisms $T(\underset{\to}{A}) \cong T(A)$ and $T(\underset{\to}{A})^{\underset{\to}{\alpha}} \cong T(A)^{\alpha}$.

The existence of a `natural' CQMS structure on the crossed product $A\rtimes_\alpha\nn$ of a $\cs$-algebraic CQMS $(A,L_A)$ by an automorphism (or tracially nondegenerate embedding) $\alpha$ is a delicate issue, one that is taken up, for example, in \cite{Kaad:2021aa}. But if $(A_X,\alpha_h)$ is a Stein system associated with $(X,h)$ such that $A=A_X$ is a $\cs$-QMBS, then a very natural choice presents itself. Actually, `all' that we require of $\alpha=\alpha_h$ is that
\begin{equation} \label{eqn:ue}
T(\underset{\to}{A}\rtimes_{\underset{\to}{\alpha}}\zz) \cong T(A)^{\alpha} \cong \prob(X)^h
\end{equation}
via the inclusion map $A\hookrightarrow \underset{\to}{A}\rtimes_{\underset{\to}{\alpha}}\zz$. This holds if $\alpha$ has finite Rokhlin dimension (see \cite[Lemma 4.8]{Jacelon:2022wr}) and more generally in the following context.

\begin{definition} \label{def:ue}
We say that a $\cs$-dynamical system $(A,\alpha)$ has the \emph{unique tracial extension property} if every $\alpha$-invariant trace on $A$ admits a unique extension to $\underset{\to}{A}\rtimes_{\underset{\to}{\alpha}}\zz$.
\end{definition}

For a complete analysis of this property for automorphic actions of discrete groups on unital $\cs$-algebras, see \cite{Ursu:2021wp}.

\begin{proposition} \label{prop:inherit}
Let $(X,\rho)$ be a compact metric space, $h \colon X \to X$ a homeomorphism and $(A_X,\alpha_h)$ a unital tracial quantum system associated with $(X,h)$, with $A=A_X$ equipped with the $\cs$-algebraic quantum metric Bauer simplex structure determined by $\rho$ (see Section~\ref{subsection:qmbs}). Suppose that $(A_X,\alpha_h)$ has the unique tracial extension property and that $A_X\rtimes_{\alpha_h}\nn$ is algebraically simple (both of which hold if $A_X$ is simple and $\alpha_h$ has finite Rokhlin dimension). Then, $A_X\rtimes_{\alpha_h}\nn$ is a $\cs$-algebraic quantum metric Choquet simplex when $T(A_X\rtimes_{\alpha_h}\nn)\cong\prob(X)^h$ is equipped with the Wasserstein metric $W_{\rho}|_{\prob(X)^h}$ (see \eqref{eqn:wass}) and the seminorm $L= L_{W_{\rho}}\colon A_X\rtimes_{\alpha_h}\nn\to[0,\infty]$ (see \eqref{eqn:L}).
\end{proposition}

\begin{proof}
We must verify the conditions of Definition~\ref{def:cqmcs}. Note that $(A\rtimes_{\alpha_h}\nn)^q$ is tracially ordered by Proposition~\ref{prop:to}. The seminorm $L$ is by definition lower semicontinuous and self adjoint,  and its kernel consists of the tracially constant elements of $A\rtimes_{\alpha_h}\nn$. Since the unique extension of $\tau\in T(\underset{\to}{A}) \cong T(A)$ to $T(\underset{\to}{A}\rtimes_{\underset{\to}{\alpha_h}}\zz) = T(A\rtimes_{\alpha_h}\nn)$ is defined by $\sum_{n\in\zz}a_nu^n\mapsto\tau(a_0)$, $L$ is also densely finite.  Since $W_{\rho}$ is convex \eqref{eqn:convex} and midpoint balanced \eqref{eqn:midpt}, so is its restriction to ${\prob(X)^h}$. As described in Remark~\ref{rem:lsc}, this guarantees that $\rho_L=W_{\rho}|_{\prob(X)^h}$ and in particular that $\rho_L$ induces the $w^*$-topology on $T(A\rtimes_{\alpha_h}\nn)$. We conclude that $(A\rtimes_{\alpha_h}\nn,L)$ is a $\cs$-algebraic quantum metric Choquet simplex.
\end{proof}

\begin{remark} \label{rem:unique}
In Section~\ref{subsection:dynamics}, we will investigate the variation of the structure described in Proposition~\ref{prop:inherit} alongside continuously varying actions $h_\theta$. One of the desired properties of a continuous field of quantum structures is suitable upper semicontinuity of Lip-norms (see Definition~\ref{def:cf}). While the issue does not arise in the examples that motivate Theorem~\ref{thm:cf}, we would like to allow for the possibility of actions that are uniquely ergodic at some fibres but not others. If $(X,h)$ is uniquely ergodic, then $A_X\rtimes_{\alpha_h}\nn$ has a unique trace, and so the formula \eqref{eqn:L} is not immediately sensible. In fact, the unique $\cs$-QMCS structure in this setting is the trivial one $L=0$. This could be problematic for upper semicontinuity. So at the cost of reducing $\ker L$ to a smaller subspace than required by Definition~\ref{def:cqmcs}, we will replace the trivial structure by the seminorm $L\colon A_X\rtimes_{\alpha_h}\nn\to[0,\infty]$ defined by $L\left(\sum_{n\in\zz}a_nu^n\right):=L_A(a_0)$. (Here, we use the same notation $L_A$ to denote the seminorm on $A=A_X$ and the corresponding one on $\underset{\to}{A}$ given by the identification $T(\underset{\to}{A}) = T(A)$.) Note in particular that if $L$ is as constructed in Proposition~\ref{prop:inherit}, then
\[
L\left(\sum_{n\in\zz}a_nu^n\right) = L_{W_{\rho}|_{\prob(X)^h}}\left(\widehat{a_0}|_{\prob(X)^h}\right)
\le L_\rho\left(\widehat{a_0}|_{X}\right) \\
= L_A\left(\sum_{n\in\zz}a_nu^n\right)
\]
which will give us the required upper semicontinuity at uniquely ergodic fibres.
\end{remark}

Recall from Remark~\ref{rem:cqmcs} that certain stably projectionless classifiable $\cs$-algebras $A$ can also be equipped with QMCS structure. If $A=A_X$ is not necessarily unital but $\alpha_h$ is invertible (so that, in particular, $A_X\rtimes_{\alpha_h}\nn \cong A_X\rtimes_{\alpha_h}\zz$), then the proof of  Proposition~\ref{prop:inherit} still demonstrates that $(A_X\rtimes_{\alpha_h}\zz, L_{W_{\rho}|_{\prob(X)^h}})$ is a $\cs$-QMCS. This provides us with dynamical $\cs$-QMCS descriptions of arbitrary metrisable Choquet simplices. 

\begin{corollary} \label{cor:dynqmcs}
Every metrisable Choquet simplex $\Delta$ arises as the trace space of a $\mathcal{W}$-stable, classifiable dynamical $\cs$-QMCS, specifically a tracial quantum crossed product associated with a minimal subshift. 
\end{corollary}

\begin{proof}
By Downarowicz's theorem \cite[Theorem 3]{Downarowicz:1991te}, there exists a minimal subshift $(X,h)$ such that $\prob(X)^h \cong \Delta$. Let $\rho$ be a compatible metric on the shift space $X$. Let $(\mathcal{A}_X,L_\rho)$ be a unital classifiable $\cs$-algebraic quantum metric Bauer simplex associated with $(X,\rho)$ as in Example~\ref{ex:cqmcs}.\ref{it:ex2}. Because $X$ is in general not connected, our construction does not guarantee that $\mathcal{A}_X$ is $K$-connected, which could be an obstruction to lifting the topological dynamical system $(X,h)$ to the $\cs$-algebraic level. We overcome this by tensoring with the monotracial, $KK$-contractible, stably projectionless $\cs$-algebra $\mathcal{W}$. This has the effect of killing $K$-theory (hence the pairing) while preserving tracial structure and classifiability. The $\cs$-algebra $\mathcal{A}_X\otimes\mathcal{W}$ is algebraically simple, $KK$-contractible, stably projectionless and classifiable, and there is an affine homeomorphism $\theta\colon T(\mathcal{A}_X) \to T(\mathcal{A}_X\otimes\mathcal{W})$ that maps $\tau$ to $\tau\otimes\tau_{\mathcal{W}}$. As discussed in Remark~\ref{rem:cqmcs}, the Cuntz--Pedersen quotient $(\mathcal{A}_X\otimes\mathcal{W})^q \cong \aff(T(\mathcal{A}_X))$ has the structure of a tracially ordered order unit space. Define the seminorm $L\colon \mathcal{A}_X\otimes\mathcal{W} \to [0,\infty]$ by $L(a)=L_\rho(\theta^*(a))$, where $\theta^*(a)(\tau)=(\theta(\tau))(a)$ for $a\in \mathcal{A}_X\otimes\mathcal{W}$ and $\tau\in T(\mathcal{A}_X)$. (Equivalently, we define a metric $\rho'$ on $T(\mathcal{A}_X\otimes\mathcal{W})$ making $\theta$ an isometry, and $L$ is the seminorm $L_{\rho'}$ defined by \eqref{eqn:lbauer}.) We can then view $(\mathcal{A}_X\otimes\mathcal{W},L)$ as a nonunital $\cs$-QMBS. Let $Y:=\partial_e(T(\mathcal{A}_X\otimes\mathcal{W})) \cong \partial_e(T(\mathcal{A}_X)) \cong X$ and let $g\colon Y \to Y$ be the homeomorphism $g:=\theta\circ h \circ \theta^{-1}$. The topological dynamical system $(Y,g)$ is isometrically conjugate to $(X,h)$, so in particular, $\prob(Y)^g \cong \prob(X)^h \cong \Delta$. Appealing to \cite[Lemma 4.7]{Jacelon:2022wr} and stably projectionless classification theory \cite[Theorem 7.5]{Elliott:2020wc} as in the proof of Theorem~\ref{thm:range} or \cite[Theorem 5.1]{Jacelon:2022wr}, we can lift $g$ to an automorphism $\alpha_g$ of $\mathcal{A}_X\otimes\mathcal{W}$ that moreover has finite Rokhlin dimension. In other words, $(\mathcal{A}_X\otimes\mathcal{W},\alpha_g)$ is a $\mathcal{W}$-stable Stein system associated with $(Y,g)$. Proposition~\ref{prop:inherit} then endows $(\mathcal{A}_X\otimes\mathcal{W})\rtimes_{\alpha_g}\zz$ with the structure of a classifiable $\cs$-QMCS whose trace space is affinely homeomorphic to $\Delta$.
\end{proof}

\subsection{Random systems} \label{subsection:random}

Here, we seize upon the present opportunity to clarify and expand on \cite[Proposition 4.3 and Example 4.10]{Jacelon:2022wr}. These observations pertain to finite-time estimates of large deviation from the spatial mean for Lipschitz observables of (quantum) dynamical systems. Specifically (in the current article's notation) it is asserted that, if $\alpha$ is an automorphism (or tracially nondegenerate endomorphism) of a $\cs$-QMBS $A$ preserving a unique trace $\tau_\alpha$ whose representing measure is $\nu=\nu_\alpha$, then for given $r \ge r_A$, nucleus $\mathcal{D}_r(A)$ and $\varepsilon>0$, there are constants $c_1,c_2>0$ such that, for every $n\in\nn$,
\begin{equation} \label{eqn:deviation}
\sup_{a\in\mathcal{D}_r(A)} \nu_\alpha \left(\left\{\tau\in \partial_e(T(A)) \mid \left| \frac{1}{n}\sum_{k=0}^{n-1} \tau(\alpha^k(a))-\tau_\alpha(a)\right| > \varepsilon \right\}\right) \le c_1e^{-c_2n\varepsilon^2}.
\end{equation}
But for \emph{deterministic} systems, this is really the same information that is provided by Birkhoff's ergodic theorem and compactness of $\mathcal{D}_r(A)$ (see Section~\ref{subsubsection:birkhoff}). Indeed, for fixed $\varepsilon$, there exists $N\in\nn$ such that for every $n\ge N$ and every $a\in\mathcal{D}_r(A)$, the set appearing on the left-hand side of \eqref{eqn:deviation} is \emph{empty}.

The situation becomes much more interesting in the setting of \emph{random} dynamical systems. More precisely, suppose that $(X,P)$ is a \emph{Markov--Feller process} (see \cite[Chapter 2]{Benoist:2016ab} or \cite{Benaim:0aa}), that is, $X$ is a compact metrisable space and $P\colon C(X) \to C(X)$ is a unital completely positive map. Equivalently, $x\mapsto P_x:=\ev_x\circ P\in C(X)^*$ is a $w^*$-continuous map from $X$ to $\prob(X)$. The process yields a random walk $(X_k)_{k\in\nn_0}$ on $X$, where $X_0=x_0$ is chosen according to some initial distribution and the next step depends only on the current position, that is, $\pp(X_{k+1}\in U \mid X_k=x) = P_x(U)$ for $n\in\nn_0$ and $U$ a Borel subset of $X$. From this point of view, the map $P$ can be described as
\[
Pf(x) = \int_X f\,dP_x = \ee(f\circ X_1 \mid X_0=x).
\]
Such a process arises, for example, when there is a family $(g_\theta)_{\theta\in\Theta}$ of continuous maps on $X$ and the parameter space $\Theta$ is equipped with a Borel probability measure $\mu$. The associated Markov--Feller operator $P=P_\mu$ is defined by $P_\mu := \mu \ast \ev_x$, where $\ast$ denotes the convolution product
\[
\mu \ast \nu(U) := \int_\Theta \nu(g_\theta^{-1}U)\, d\mu(\theta), \quad \text{$U\subseteq X$ Borel}.
\]
Some such systems are uniquely ergodic: there is a unique $P_\mu$-invariant measure, that is, a unique $\nu\in\prob(X)$ such that $\int_X Pf\,d\nu = \int_X f\,d\nu$ for every $f\in C(X)$ (or equivalently, $\nu$ is the unique $\mu$-stationary measure, meaning that $\mu \ast \nu = \nu$). This holds, for example, for:
\begin{itemize}
\item Lipschitz actions on $X$ that are `contracting on average' (see \cite[Theorem 1.1]{Diaconis:1999aa} and also \cite[Example 11.4]{Benoist:2016ab}, which describes a process where at each step, one of two contractive maps on the Cantor set is chosen with equal probability);
\item actions on projective space $\pp(\rr^d)$ of `proximal' subsemigroups of $\mathrm{SL}(d,\rr)$, as well as actions  on certain flag varieties of certain reductive groups (see \cite[Chapters 4,13]{Benoist:2016ab}).
\end{itemize}
Breiman's law of large numbers \cite{Breiman:1960aa} says that, in these cases, for every $f\in C(X)$ and any initial position $X_0=x_0\in X$,
\begin{equation} \label{eqn:breiman}
\lim_{n\to\infty} \frac{1}{n}\sum_{k=0}^{n-1} f(x_k) = \int_X f\,d\nu 
\end{equation}
for $\pp_{x_0}$-almost every trajectory $(y_k)_{k=0}^\infty\in X^{\nn_0}$, where $\pp_{x_0}$ denotes the Markov measure associated with the starting point $x_0$ and process $P$ (that is, the probability that a trajectory lies in a Borel subset $U$ of $X^{\nn_0}$ is $\pp_{x_0}(U)$).

This phenomenon continues to hold in the noncommutative tracial setting, that is, if $A$ is a  $\cs$-QMBS and $(\alpha_\theta)_{\theta\in\Theta}$ is a random family of tracially nondegenerate endomorphisms of $A$ such that $(T(\alpha_\theta)|_{\partial_e(T(A))})_{\theta\in\Theta}$ is a uniquely ergodic Markov--Feller process on $X:=\partial_e(T(A))$. (Such noncommutative systems certainly do exist. For example, $A$ might be $K$-connected and classifiable, and $(\alpha_\theta)_{\theta\in\Theta}$ could be the lift of a uniquely ergodic system on $X$). We are interested in the long-term tracial behaviour of the random system $(\alpha_\theta)_{\theta\in\Theta}$. Given an initial observable $[a_0]\in A^q$, we let $\pp_{[a_0]}$ denote the corresponding Markov measure on $\ell_\infty A^q$ equipped with the product Borel $\sigma$-algebra. What follows is is a finite-time estimate of large deviation associated with \eqref{eqn:breiman}. Although stated under the assumption that the transformations $T(\alpha_\theta)$ are $1$-Lipschitz, the proposition should also hold for endomorphisms $\alpha_\theta$ that are tracially `contracting on average' in a suitable sense (see \cite[Chapter 11]{Benoist:2016ab}).

\begin{proposition} \label{prop:breiman}
Let $(A,L)$ be a $\cs$-QMBS, let $(\Theta,\mu)$ be a probability space and let $(\alpha_\theta)_{\theta\in\Theta}$ be a family of tracially nondegenerate endomorphisms of $A$. Suppose that the continuous affine maps $T(\alpha_\theta)_{\theta\in\Theta}$ are $1$-Lipschitz on $\partial_e(T(A))$ (hence on all of $T(A)$) and that the associated Markov--Feller system $(\partial_e(T(A)),P_\mu)$ admits a unique $P_\mu$-invariant measure $\nu\in\prob(\partial_e(T(A)))$. Let $\tau_\nu\in T(A)$ be the trace whose representing measure is $\nu$ and fix a nucleus $\mathcal{D}_r(A)$. Then, for every $\varepsilon>0$, there are constants $c_1,c_2>0$ such that, for every $\tau\in T(A)$, $n\in\nn$ and $a_0\in \mathcal{D}_r(A)$,
\begin{equation} \label{eqn:breidev}
\pp_{[a_0]} \left(\left\{([a_k])_{k=0}^\infty\in \ell_\infty A^q \mid \left| \frac{1}{n}\sum_{k=0}^{n-1} \tau(a_k)-\tau_\nu(a_0)\right| > \varepsilon \right\}\right) \le c_1e^{-c_2n\varepsilon^2}.
\end{equation}
\end{proposition}

\begin{proof}
By \cite[Proposition 3.1]{Benoist:2016ur}, the inequality \eqref{eqn:breidev} with $\frac{\varepsilon}{2}$ in place of $\varepsilon$ (to make room for later estimates) holds for fixed $a_0$ and every $\tau\in\partial_e(T(A))$ . The bound displayed in \cite[Proposition 3.1]{Benoist:2016ur} is of the form $c_1'e^{-Mn}$, but one sees from the proof that $M=c_2\varepsilon^2$ for a suitable constant $c_2$. In fact, one also sees that, by compactness of $\mathcal{D}_r(A)$ (or rather, its image in $C(\partial_e(T(A)))$), the same constants can be used uniformly across $\mathcal{D}_r(A)$. Finally, approximating $\tau\in T(A)$ by a convex combination of elements of $\partial_e(T(A))$, one sees that \eqref{eqn:breidev} holds for every $\tau\in T(A)$. Note that, to succeed with this approximation, we should ensure that we can use a common large set of `good' trajectories $(\alpha_k)_{k=1}^\infty \in \{\alpha_\theta \mid \theta\in\Theta\}^\nn$ that works for every $\tau\in\partial_e(T(A))$, that is, trajectories for which
\begin{equation} \label{eqn:commongood}
\sup_{a_0\in\mathcal{D}_r(A)} \sup_{\tau \in T(A)} \sup_{n\in\nn} \left|\frac{1}{n}\sum_{k=0}^{n-1} \tau(a_k)-\tau_\nu(a_0)\right| \le \frac{\varepsilon}{2},
\end{equation}
where $a_k=\alpha_{k} \circ \dots \circ \alpha_1(a_0)$. But we can indeed accomplish this by compactness of $\partial_e(T(A))$ and $\mathcal{D}_r(A)$. More precisely, let $\tau_1,\dots,\tau_m \in \partial_e(T(A))$ be an $\frac{\varepsilon}{4}$-net for  $(\partial_e(T(A)),\rho_L)$, so that for every $\tau\in \partial_e(T(A))$, there is some $i\in[1,m]$ such that
\[
|\tau(a)-\tau_i(a)| \le L(a)\rho_L(\tau,\tau_i) < \frac{\varepsilon}{4} \: \text{ for every } a \in \mathcal{D}_r(A).
\]
Then, every $\tau\in T(A)$ can be approximated on $\mathcal{D}_r(A)$ up to $\frac{\varepsilon}{2}$ by a finite convex combination $\sum_{i=1}^m\lambda_i\tau_i=:\tau'$. The Markov probability of the intersection of the sets of trajectories for which \eqref{eqn:commongood} holds for $\tau=\tau_i$, $i\in[1,m]$, is at least $1-c_1e^{-c_2n\varepsilon^2}$ with $c_1:=mc_1'$. On this intersection, for every $a_0\in\mathcal{D}_r(A)$ and $n\in\nn$, we have
\begin{align*}
\left|\frac{1}{n}\sum_{k=0}^{n-1} \tau(a_k)-\tau_\nu(a_0)\right| &\le \frac{1}{n}\sum_{k=0}^{n-1} |\tau(a_k)-\tau'(a_k)| + \left|\frac{1}{n}\sum_{k=0}^{n-1} \tau'(a_k)-\tau_\nu(a_0)\right|\\
&\le \frac{\varepsilon}{2} + \left|\frac{1}{n}\sum_{k=0}^{n-1} \sum_{i=1}^m\lambda_i\tau_i(a_k) - \sum_{i=1}^m\lambda_i\tau_\nu(a_0)\right|\\
&\le \frac{\varepsilon}{2} + \sum_{i=1}^m\lambda_i\left|\frac{1}{n}\sum_{k=0}^{n-1} \tau_i(a_k)-\tau_\nu(a_0)\right|\\
&\le \varepsilon. \qedhere
\end{align*}
\end{proof}

\section{Quantum fields} \label{section:cf}

In this section, we address the question: how does the geometry of a $\cs$-QMCS $(A,L)$ change in conjunction with deformation of its quantum metric structure? We will specifically examine continuous variation of:
\begin{enumerate}[(I)]
\item \label{item:metric} the seminorm $L$ (or metric $\rho_L$);
\item \label{item:dynamics} the underlying dynamics $(X,h)$ when $A$ is constructed as a tracial quantum crossed product $A_X\rtimes_{\alpha_h}\nn$.
\end{enumerate}

Our motivating example for \eqref{item:metric} is that of a vibrating piano string. We think of the interval $I=[0,\pi]$ as $\partial_e(T(A))$ for a $\cs$-QMBS $A$, and we encode the motion of the string by adjusting the metric on $I$  as
\begin{equation} \label{eqn:wave}
\rho_t(a,b) = \int_a^b \sqrt{1+\left(\frac{\partial}{\partial x}u(x,t)\right)^2}\,dx,
\end{equation}
where $u(x,t)$ is the solution over some time interval $[0,p]$ to the appropriate wave equation. That is, $\rho_t(a,b)$ is the arc length from $a$ to $b$ along the stretched string at time $t$. The example to bear in mind for \eqref{item:dynamics} is an isometric action on a Riemannian manifold that is conjugated by an isotopy, in particular, rotation $h_\theta$ of the circle $S^1$ deformed to $g_t \circ h_\theta \circ g_t^{-1}$ via a path of diffeomorphisms (or uniformly bi-Lipschitz homeomorphisms) $g_t \colon S^1 \to S^1$. In both cases, we will assemble the families of structures into continuous fields of quantum spaces. In the former case, we investigate the variation of unitary distances $d_U$ (see Section~\ref{subsubsection:unitary}) and Birkhoff convergence rates $\beta_r$ (see Section~\ref{subsubsection:birkhoff}). In the latter, we are interested in the quantum Gromov--Hausdorff distance between fibres (see Section~\ref{section:qgh}).

With these examples in mind, we adopt from \cite[Definition 6.4]{Li:2006aa} the following notion of a continuous field of quantum spaces. For background on continuous fields of Banach spaces and $\cs$-algebras, see \cite[Chapter 10]{Dixmier:1964rt}.

\begin{definition} \label{def:cf}
Let $\Theta$ be a compact Hausdorff topological space. We say that $((A_\theta,L_\theta)_{\theta\in\Theta},\Gamma)$ is a \emph{continuous field of $\cs$-algebraic quantum metric Choquet simplices over $\Theta$} if:
\begin{enumerate}[1.]
\item \label{it:cf1} each $(A_\theta,L_\theta)$ is a $\cs$-algebraic quantum metric Choquet simplex;
\item \label{it:cf2} $((A_\theta)_{\theta\in\Theta},\Gamma)$ is a continuous field of $\cs$-algebras that is unital in the sense that $(\theta\mapsto 1_{A_\theta}) \in \Gamma$;
\item \label{it:cf3} for every $\theta\in\Theta$, $a\in L_\theta^{-1}([0,\infty))\subseteq A$ and $\varepsilon>0$, there is a Lipschitz section $f$ at $\theta$ (that is, $f\in\Gamma$ such that the map $\eta\mapsto L_\eta(f_\eta)$ is upper semicontinuous at $\theta$) with $\|f_\theta-a\|<\varepsilon$ and $L_{\theta}(f_{\theta}) < L_{\theta}(a) + \varepsilon$.
\end{enumerate}  
\end{definition}

\subsection{Metric deformation}

For the following (a continuous-field version of Theorem~\ref{thm:nucleus}), recall the definition of the unitary distance \eqref{eqn:unitary} and Birkhoff convergence rate (Definition~\ref{def:birkhoff}) relative to a nucleus $\mathcal{D}_r(A)$ of a $\cs$-QMCS $(A,L)$. 

\begin{theorem} \label{thm:metricdeform}
Let $\Theta$ be a compact Hausdorff space and $(\rho_\theta)_{\theta\in\Theta}$ a family of metrics on a compact metrisable space $X\ne\{\mathrm{pt}\}$ that varies Lipschitz continuously in the following sense: there are continuous functions $k,K\colon\Theta\times\Theta\to(0,\infty)$ that are both constantly $1$ on the diagonal and otherwise satisfy 
\begin{equation} \label{eqn:lip}
k_{\eta,\theta} \rho_{\eta}(x,y) \le \rho_{\theta}(x,y) \le K_{\eta,\theta}\rho_{\eta}(x,y) \:\text{ for every }\: \eta,\theta\in\Theta,\,x,y\in X.
\end{equation}
Let $A$ be a separable, unital $\cs$-algebra with $\partial_e(T(A))\cong X$, and for each $\theta\in\Theta$, let $L_\theta$ be the seminorm corresponding to $\rho_\theta$ via \eqref{eqn:lbauer} and making $(A,L_\theta)$ into a $\cs$-algebraic quantum metric Bauer simplex. Then, $((A,L_\theta)_{\theta\in\Theta},C(\Theta,A))$ is a continuous field of $\cs$-algebraic quantum metric Choquet simplices. Moreover, for every $r\ge\sup_{\theta\in\Theta}r_{(X,\rho_\theta)}$, there is a family $(\mathcal{D}_{r,\theta}(A))_{\theta\in\Theta}$ of nuclei of $(A,L_\theta)$ such that:
\begin{enumerate}[(i)]
\item \label{it:wave2} for any tracially nondegenerate $\alpha\in\End(A)$ with a unique invariant trace, and any $\varepsilon>0$, the map sending $\theta\in\Theta$ to $\beta_{r,\theta}(\varepsilon)$ is upper semicontinuous (where $\beta_{r,\theta}$ denotes the Birkhoff convergence rate relative to $\mathcal{D}_{r,\theta}(A)$);
\item \label{it:wave1} for any unital $\cs$-algebra $B$ and $\varphi,\psi\in\emb(A,B)$, the map that sends $\theta\in\Theta$ to $d_U(\varphi,\psi)|_{\mathcal{D}_{r,\theta}(A)}$ is continuous.
\end{enumerate}
\end{theorem}

\begin{proof}
The assumption \eqref{eqn:lip} ensures that, for every $\eta,\theta\in\Theta$,
\begin{equation} \label{eqn:include}
k_{\eta,\theta} \lip(X,\rho_\eta) \subseteq \lip(X,\rho_\theta) \subseteq K_{\eta,\theta}\lip(X,\rho_\eta)
\end{equation}
and therefore that the corresponding seminorms $L_\eta,L_\theta$ defined by \eqref{eqn:lbauer} satisfy
\[
k_{\eta,\theta} L_\theta \le L_\eta \le  K_{\eta,\theta} L_\theta.
\]
It follows that $L_\theta^{-1}([0,\infty))$ is independent of $\theta$, and that the map $\theta\mapsto L_\theta(a)$ is continuous for every $a\in A$. This verifies condition~\ref{it:cf3} of Definition~\ref{def:cf} (with $f$ the constant section $\eta\mapsto a$) and therefore shows that we have a continuous $\cs$-QMCS field.

We also have that the map sending $\theta\in\Theta$ to the radius $r_{(X,\rho_\theta)} = r_{(\prob(X),W_{\rho_\theta})}$ is continuous and therefore bounded. For $r \ge \sup_{\theta\in\Theta}r_{(X,\rho_\theta)}$ and $\theta\in\Theta$, let $\mathcal{L}_{r,\theta}$ be the nucleus of $(C(X),L_\theta)$ defined in \eqref{eqn:drcomm}, that is,
\[
\mathcal{L}_{r,\theta} = \{f \in C(X)_{sa} \mid \|f\|\le r,\: |f(x)-f(y)|\le\rho_\theta(x,y)\:\text{ for every } x,y\in X\}.
\]
The map sending $\theta\in\Theta$ to $\mathcal{L}_{r,\theta}$ is continuous with respect to the Hausdorff metric on subsets of $(C(X)),\|\cdot\|)$, where $\|\cdot\|$ is the uniform norm. Indeed, if $f=f_+-f_-\in\mathcal{L}_{r,\theta}$ (that is, writing $f$ as the sum of its positive and negative parts), then by \eqref{eqn:include},
\[
\min\left\{\frac{f_+}{K_{\eta,\theta}},r\right\} - \min\left\{\frac{f_-}{K_{\eta,\theta}},r\right\} \in \mathcal{L}_{r,\eta}.
\]
Noting as in the proof of Proposition~\ref{prop:dirichlet} that every $f\in C(X)$ extends uniquely to an element of $\aff(\prob(X))$ without increase of uniform norm or Lipschitz seminorm, we get a corresponding family of nuclei $\mathcal{D}^q_{r,\theta}(A)$ with compact union in $\aff(\prob(X))\cong A^q$. In other words, every element of $\mathcal{D}^q_{r,\theta}(A)$ is the unique continuous affine  extension of an element of $\mathcal{L}_{r,\theta}$ from the boundary $X$ to the whole simplex $\prob(X)$. Let $s \colon A^q \to A_{sa}$ be  the Michael--Bartle--Graves inverse to the quotient map $q \colon A_{sa} \to A^q$ as in the proof of Theorem~\ref{thm:nucleus}. The desired family of nuclei is
\[
\mathcal{D}_{r,\theta}(A) := s(\mathcal{D}^q_{r,\theta}(A)).
\]
The asserted properties of $d_U$ and $\beta_{r,\theta}(\varepsilon)$ hold because the map $\theta \mapsto \mathcal{D}_{r,\theta}(A)$ is Hausdorff-continuous. For \eqref{it:wave2}, suppose that $\theta_i\to\theta\in\Theta$ and that for every $i$, $\beta_{r,\theta_i}(\varepsilon) \ge N$, that is, there exists $n_i\ge N-1$ such that
\[
\Delta_{\theta_i}(n_i) := \sup_{a\in\mathcal{D}_{r,\theta_i}(A)} \sup_{\tau\in T(A)} \left| \frac{1}{n_i}\sum_{k=0}^{n_i-1} \tau(\alpha^k(a))-\tau_\alpha(a)\right| > \varepsilon.
\]
Then, for any sufficiently large $i$ we have
\begin{align*}
\Delta_{\theta}(n_i) &= \sup_{a\in\mathcal{D}_{r,\theta}(A)} \sup_{\tau\in T(A)} \left| \frac{1}{n_i}\sum_{k=0}^{n_i-1} \tau(\alpha^k(a))-\tau_\alpha(a)\right|\\
&\ge \Delta_{\theta_i}(n_i) - 2\dist_{\mathrm{H}}^{\|\cdot\|}\left(\mathcal{D}_{r,\theta}(A),\mathcal{D}_{r,\theta_i}(A)\right)\\
&> \varepsilon,
\end{align*}
so $\beta_{r,\theta}(\varepsilon) \ge N$ as well.

For \eqref{it:wave1}, suppose that $\varphi,\psi \colon A \to B$ are unital embeddings. The situation for the unitary distance
\[
d_U(\varphi,\psi)|_{\mathcal{D}_{r,\theta}(A)} = \inf_{u\in\mathcal{U}(B)} \sup_{a\in\mathcal{D}_{r,\theta}(A)} \|\varphi(a)-u\psi(a)u^*\|
\]
is similar, because
\[
\left| d_U(\varphi,\psi)|_{\mathcal{D}_{r,\eta}(A)} - d_U(\varphi,\psi)|_{\mathcal{D}_{r,\theta}(A)} \right| \le 2\dist_{\mathrm{H}}^{\|\cdot\|}\left(\mathcal{D}_{r,\eta}(A),\mathcal{D}_{r,\theta}(A)\right). \qedhere
\]
\end{proof}

\begin{example} \label{ex:metricdeform1}
Let us check that the family \eqref{eqn:wave} of metrics corresponding to the wave equation satisfies \eqref{eqn:lip}. (While the observations that follow certainly apply more generally, the wave equation is an illuminating example to keep in mind.) It suffices to find continuous functions $m,M\colon[0,p]\to(0,\infty)$ satisfying $m_t\rho_0 \le \rho_t \le M_t\rho_0$, for we can then take $k_{s,t}=\frac{m_t}{M_s}$ and $K_{s,t}=\frac{M_t}{m_s}$. We set $m_t:=\inf_{a \ne b\in I}\frac{\rho_t(a,b)}{\rho_0(a,b)}$ and $M_t:=\sup_{a \ne b\in I}\frac{\rho_t(a,b)}{\rho_0(a,b)}$. Note that, by L'Hospital's rule and the fundamental theorem of calculus,
\[
\lim_{b\to a} \frac{\rho_t(a,b)}{\rho_0(a,b)} = \left.\sqrt{\frac{1+\left(\frac{\partial}{\partial x}u(x,t)\right)^2}{1+\left(\frac{\partial}{\partial x}u(x,0)\right)^2}}\right|_{x=a},
\]
so $m_t$ and $M_t$ take values in $(0,\infty)$ and
\[
m_t\rho_0(a,b) \le \rho_t(a,b) \le M_t\rho_0(a,b) \: \text{ for every } a,b\in I.
\]
Their continuity follows from continuity of $t\mapsto\rho_t(a,b)$, which in turn follows from continuity of $t\mapsto\frac{\partial}{\partial x}u(x,t)$.

As for a model $\cs$-QMBS $A$ witnessing the effects described in Theorem~\ref{thm:metricdeform}, see Example~\ref{ex:cqmcs}. In particular, continuous families of unitary distances as in \eqref{it:wave1} can be associated to dimension drop algebras and to the algebra $\mathcal{A}_I$ (constructed in \cite[Section 4.4.4]{Jacelon:2021vc} as a limit of prime dimension drop algebras, and generalised in \cite[Theorem 4.4]{Jacelon:2022wr}). Any uniquely ergodic continuous self-map $h$ of the interval $I$ can be lifted to a tracially nondegenerate endomorphism $\alpha_h$ of $\mathcal{A}_I$ that fixes a unique trace. The Birkhoff convergence rates of $\alpha$ at different times $t$ can be measured and compared, and will vary upper semicontinuously for any fixed $\varepsilon$ as in \eqref{it:wave2}. One example of this procedure would be to take $h$ to be constant, or in other words $\alpha_h$ to be trace-collapsing. This is uninteresting because we would then have $\beta_{r,t}(\varepsilon)=1$ for every $\varepsilon>0$. Further, had we been hoping for an \emph{automorphism} $\alpha$, then our fate would have been disappointment, for there are no uniquely ergodic homeomorphisms of the interval. That is unless we are willing to first identify the endpoints of the string $I$ to obtain a loop $S^1$. The vibration of the loop is encoded by the metrics
\[
\rho'_t(a,b):=\min\{\rho_t(a,b),\rho_t(a,0)+\rho_t(p,b)\},
\]
relative to which we can compute the convergence rates $\beta_{r,t}$ associated with uniquely ergodic circle maps (say, irrational rotation) lifted to $\mathcal{A}_{S^1}$.
\end{example}

\begin{example}  \label{ex:metricdeform2}
We recall from \cite[Section 2.2]{Jacelon:2021vc} that the \emph{transport constant} $k_Z$ of a compact, path-connected metric space $(Z,\rho)$ is a measurement of the extent to which the geometry of $Z$ permits efficient continuous redistributions of masses represented by probability measures. Here, `efficient' means relative to the \emph{$\infty$-Wasserstein distance} $W_\infty$ (see \cite[Section 2.1]{Jacelon:2021vc}). Just like the $1$-Wasserstein distance \eqref{eqn:wass}, it varies continuously as $(\rho_\theta)_{\theta\in\Theta}$ varies according to \eqref{eqn:lip}. The same is true of $k_Z$. So in these circumstances, \cite[Theorem 3.1]{Jacelon:2021vc} provides (for certain $Z$ and $B$) a continuous family of solutions to the quantised, $K$-localised optimal transport problem in $\emb(Z,B)$. More precisely, suppose that $k_{Z,\theta}<\infty$ for some (hence every) $\theta$, that $K^*(Z)$ is finitely generated and torsion free, and that $B$ is a unital classifiable $\cs$-algebra of real rank zero whose tracial boundary $\partial_e(T(B))$ is compact and finite dimensional. Then, for every $\varphi\in\emb(C(Z),B)$, there is a continuous family of Czech--Kazakh inequalities
\[
W_{\infty,\theta}(\varphi,\psi) \le d_U(\varphi,\psi)|_{\mathcal{D}_{r,\theta}(C(Z))} \le k_{Z,\theta} \cdot W_{\infty,\theta}(\varphi,\psi)
\]
in
\begin{equation} \label{eqn:klocal}
\emb(C(Z),B)_{\varphi} := \{\psi \in\emb(C(Z),B) \mid K_*(\psi) = K_*(\varphi)\}.
\end{equation}
\end{example}

\subsection{Dynamical deformation} \label{subsection:dynamics}

\begin{proposition} \label{prop:rokhlin}
Let $\Theta$ be a compact metrisable space, $A$ a unital classifiable $\cs$-algebra and $(\alpha_\theta)_{\theta \in \Theta}$ a family of unital endomorphisms of $A$ that varies continuously with respect to the topology of pointwise convergence. Then, there is a continuous family of unital endomorphisms $\theta\mapsto\alpha'_\theta$ such that, for every $\theta\in\Theta$, $\Ell(\alpha'_\theta)=\Ell(\alpha_\theta)$ and $\alpha'_\theta$ has finite Rokhlin dimension (so in particular, $A\rtimes_{\alpha'_\theta}\nn$ is classifiable).
\end{proposition}

\begin{proof}
The endomorphisms $(\alpha_\theta)_{\theta\in\Theta}$ provide an endomorphism $\alpha$ of $C(\Theta,A)$, namely, $\alpha(f)(\theta):=\alpha_\theta(f(\theta))$. By the results of \cite[Section 11]{Szabo:2019te} (signposted in \cite[Lemma 4.7]{Jacelon:2022wr}), there is a unital endomorphism $\alpha'=(\alpha'_\theta)_{\theta\in\Theta}$ of $C(\Theta,A)$ whose Rokhlin dimension is $\le 1$ and whose Elliott invariant is equal to that of $\alpha$. (This only uses separability and $\js$-stability of $C(\Theta,A) \cong C(\Theta) \otimes A$.) The asserted properties then hold for $\alpha'_\theta\in\End(A)$.
\end{proof}

We recall from \cite[Definition 6.8]{Fukaya:1990aa} the definition of the equivariant Gromov--Hausdorff distance between group actions.

\begin{definition} \label{def:egh}
The \emph{equivariant Gromov--Hausdorff distance} $\dist_{\mathrm{GH}}(\alpha_0,\alpha_1)$ between actions $\alpha_0$ and $\alpha_1$ of a topological group $G$ on metric spaces $(X_0,\rho_0)$ and $(X_1,\rho_1)$ is the infimum over all $\varepsilon>0$ such that, for $i\in\{0,1\}=\zz/2\zz$, there is a map $f_i\colon X_i\to X_{i+1}$ that satisfies
\begin{equation} \label{eqn:egh}
\sup_{x\in X_{i+1}} \rho_i((\alpha_{i+1}(g)\circ f_i)(x),(f_i\circ\alpha_i(g))(x)) < \varepsilon \:\text{ for every }\: g\in G
\end{equation}
and whose image is $\varepsilon$-dense in $X_{i+1}$.
\end{definition}

A similar (weaker) distance can be defined relative to generating sets $S$ of $G$ by asking for \eqref{eqn:egh} to hold only for $g\in S$ (see \cite[Definition 3.1]{Chung:2020aa}), but we will require the rather strong $S=G$ topology on the set of isometry classes of $G$-spaces. Some examples of $\dist_{\mathrm{GH}}$-convergent sequences of isometric actions are provided in \cite[Examples 4.6--4.8]{Chung:2020aa}. We discuss other examples in Example~\ref{ex:rotation} and Example~\ref{ex:sphere}.\\

To the extent that it ties together its narrative threads, the following analogue of \cite[Theorem C]{Kaad:2021aa} is the present article's climax. Its proof hinges upon the main result of \cite{Carrion:2024aa}, which provides a classification via $\underline{K}T_u$ of unital, full, nuclear embeddings of unital, separable, exact, UCT $\cs$-algebras $A$ into unital, $\js$-stable $\cs$-algebras $B$ that have strict comparison. (This theorem also has a nonunital version in \cite{Carrion:2024aa}, but we only consider unital $\cs$-algebras in Theorem~\ref{thm:cf}.) As observed in \cite{Carrion:2024aa}, this in particular applies to unital, classifiable $A$ (that is, $A$ simple, separable, nuclear, $\js$-stable and satisfying the UCT) and  $B=C(\Theta,A)$ (already observed in Proposition~\ref{prop:rokhlin} to be $\js$-stable, and explicitly shown in \cite{Carrion:2024aa} to have strict comparison).

\begin{theorem} \label{thm:cf}
Let $(X,\rho)$ be a compact, connected metric space, $\Theta$ a compact metrisable space and $(h_\theta)_{\theta\in\Theta}$ a family of pointwise continuously varying maps $X\to X$ that are uniformly Lipschitz, meaning that there exists $M>0$ such that
\[
\rho(h_\theta^n(x),h_\theta^n(y)) \le M \rho(x,y) \: \text{ for every } \theta\in\Theta,\,n\in \zz,\,x,y\in X. 
\]
Then, there exists a unital, classifiable $\cs$-algebraic quantum metric Bauer simplex $A=A_X$ associated with $(X,\rho)$, and a family of tracial quantum systems $(A_X,\alpha_\theta)_{\theta\in\Theta}$ associated with $(X,h_\theta)_{\theta\in\Theta}$, such that:
\begin{enumerate}[(i)]
\item \label{it:field0} every $\alpha_\theta$ is approximately unitarily equivalent to an automorphism;
\item \label{it:field1} $\theta\mapsto\alpha_\theta$ is continuous with respect to the topology of pointwise convergence and so yields an endomorphism $\alpha$ of $C(\Theta,A_X)$;
\item \label{it:field2} $((A_X\rtimes_{\alpha_\theta}\nn)_{\theta\in\Theta},C(\Theta,A_X)\rtimes_\alpha\nn)$ forms a continuous field of $\cs$-algebras;
\item \label{it:field3} every $\alpha_\theta$ has finite Rokhlin dimension, so in particular, $(A_X,\alpha_\theta)$ has the unique tracial extension property and $A_X\rtimes_{\alpha_\theta}\nn$ is classifiable.
\end{enumerate}
For every $\theta\in\Theta$, let $L_\theta$ be the seminorm on $A_X\rtimes_{\alpha_\theta}\nn$ described in Proposition~\ref{prop:inherit}. Suppose further that $\theta\mapsto h_\theta$ is continuous with respect to the equivariant Gromov--Hausdorff topology. Then, $((A_X\rtimes_{\alpha_\theta}\nn,L_\theta)_{\theta\in\Theta},C(\Theta,A_X)\rtimes_\alpha\nn)$ forms a continuous field of $\cs$-algebraic quantum metric Choquet simplices and the the quantum distance functions
\begin{equation} \label{eqn:qdf}
d_\theta\colon\eta\mapsto\dq((A_X\rtimes_{\alpha_\eta}\nn,L_\eta),(A_X\rtimes_{\alpha_\theta}\nn,L_\theta))
\end{equation}
are continuous on $\Theta$.
\end{theorem}

\begin{proof}
The $\cs$-QMBS $A$ can be taken to be one of the algebras $\mathcal{A}_X$ described in Example~\ref{ex:cqmcs}.\ref{it:ex2}. By \cite[Proposition 3.5]{Bosa:aa}, $T(C(\Theta,A))$ is a Bauer simplex with
\begin{equation} \label{eqn:product}
\partial_e(T(C(\Theta,A))) \cong \partial_e(T(C(\Theta))) \times \partial_e(T(A)) \cong \Theta \times X.
\end{equation}
It follows from \eqref{eqn:product} that the map $(\theta,x) \mapsto h_\theta(x)$, which by assumption is continuous, extends to a continuous affine map $h\colon T(C(\Theta,A)) \to T(A)$. By $K$-connectedness of $A$, $h$ is compatible with the maps $\iota_* \colon K_*(A) \to K_*(C(\Theta,A))$ induced by the inclusion $\iota$ of $A$ into $C(\Theta,A)$ as constant sequences (that is, $\iota(a) = 1\otimes a$), so we obtain a morphism $(h_*,\iota_*) \colon KT_u(A) \to KT_u(C(\Theta,A))$ in the notation of \cite{Carrion:wz} (see \cite[Definition 2.3]{Carrion:wz}). By \cite [Theorem 3.9]{Carrion:wz}, this morphism can be extended (non-canonically) to a morphism of the total invariant $\underline{K}T_u$. By \cite{Carrion:2024aa} (as discussed before the statement of Theorem~\ref{thm:cf}), this can be lifted to a unital embedding $\alpha = (\alpha_\theta)_{\theta\in\Theta} \colon A \to C(\Theta,A)$, thus giving us \eqref{it:field1}. As the invariant of every $\alpha_\theta$ is an isomorphism, \eqref{it:field0} holds by \cite [Theorems 9.3, 9.6]{Carrion:wz}. By \cite[Corollary 3.6]{Rieffel:1989ly}, $((A\rtimes_{\alpha_\theta}\nn)_{\theta\in\Theta},C(\Theta,A)\rtimes_\alpha\nn)$ forms a continuous field of $\cs$-algebras, so we have \eqref{it:field2}. Note that, while \cite[Corollary 3.6]{Rieffel:1989ly} is stated for actions by automorphisms, the key point, as spelled out in \cite[Theorem 3.5]{Rieffel:1989ly} and its preceding discussion, is that there should be a family of covariant representations, varying continuously over the fibres, on a common Hilbert space. This ensures lower semicontinuity, and upper semicontinuity holds automatically as pointed out in \cite[Theorem 3.2]{Rieffel:1989ly}. In any event, the common Hilbert space for the fibres is in the present situation provided by
\[
A\rtimes_{\alpha_\theta}\nn = \underset{\to}{A}\rtimes_{\underset{\to}{\alpha_\theta}}\zz \hookrightarrow \underset{\to}{C(\Theta,A)}\rtimes_{\underset{\to}{\alpha}}\zz.
\]
Next, by Proposition~\ref{prop:rokhlin}, we can ensure that each $\alpha_\theta$ has finite Rokhlin dimension, giving us \eqref{it:field3}. 

Finally, suppose that $\theta\mapsto h_\theta$ is continuous with respect to the equivariant Gromov--Hausdorff topology. We will appeal to \cite[Theorem 1.3]{Chung:2020aa}, a result which is stated for \emph{isometric} actions but whose proof works just as well for group actions that are uniformly Lipschitz (at the negligible cost of increasing some estimating constants by a bounded factor). Let $x=\sum_{n\in\zz}a_nu_\theta^n \in \underset{\to}{A}\rtimes_{\underset{\to}{\alpha_\theta}}\zz = A\rtimes_{\alpha_\theta}\nn$. Viewing the coefficients $a_n$ as constant functions $\Theta\to \underset{\to}{A}$,
\[
f := \sum_{n\in\zz}a_nu^n \in \underset{\to}{C(\Theta,A)}\rtimes_{\underset{\to}{\alpha}}\zz = C(\Theta,A)\rtimes_\alpha\nn
\]
is a continuous section with $f_\theta=x$. By definition, $L_\eta(f_\eta)$ is the Lipschitz seminorm of $\widehat{a_0}|_{\prob(X)^{h_\eta}}$ (see Proposition~\ref{prop:inherit} and the discussion that precedes it). By \cite[Theorem 1.3]{Chung:2020aa}, the map $\eta \mapsto \prob(X)^{h_\eta}$ is continuous with respect to the Hausdorff metric associated with $(\prob(X),W_\rho)$, and so $\eta\mapsto L_\eta(f_\eta)$ is also continuous. Except, as forewarned in Remark~\ref{rem:unique}, we should be careful at values $\theta$ for which $h_\theta$ is uniquely ergodic. At these points, $L_\eta(f_\eta)$ is upper semicontinuous, which is enough to fulfill the requirements of Definition~\ref{def:cf}. We therefore have a continuous field of $\cs$-algebraic quantum metric Choquet simplices. Continuity of the functions $d_\theta$ follows from Proposition~\ref{prop:gh}.
\end{proof}

Theorem~\ref{thm:cf} in particular applies to deformations $h_t:=g_t \circ h \circ g_t^{-1}$ of an isometric action $h$ on a compact, connected Riemannian manifold $(X,\rho)$ by a smooth path of diffeomorphisms $g_t\colon X\to X$. These maps are uniformly Lipschitz and (by design) vary continuously with respect to the equivariant Gromov--Hausdorff distance. Indeed,
\[
t \mapsto \prob(X)^{g_t \circ h \circ g_t^{-1}} = (g_t)_*\left(\prob(X)^h\right)
\]
is Hausdorff continuous (regardless of whether or not the maps $g_t$ are Lipschitz). This affords us access to a range of examples.

\begin{example}[Tracial rotation algebras] \label{ex:rotation}
Equip the circle $S^1$ with the geodesic metric, let $h_\theta \colon S^1 \to S^1$ be rotation by some fixed angle $2\pi\theta\in(0,2\pi)$, and fix a path of diffeomorphisms $g_t \colon S^1 \to S^1$ that varies continuously relative to the Whitney $C^1$ -topology (see \cite[Section 8]{Kaad:2021aa}), so that, in particular, $g_t$ and its derivative $g'_t$ vary uniformly continuously. For example, let us take $(g_t)_{t\in[-1,1]}$ defined by $g_t(x) = x + \frac{t}{2} \sin^2x$ (viewing $S^1$ as the interval $[0,\pi]$ with its endpoints identified). Then, $(g_t)_{t\in[-1,1]}$ is uniformly (bi-)Lipschitz, and hence so is the family of actions of $\zz$ on $S^1$ defined by $n \mapsto g_t \circ h_\theta^n \circ g_t^{-1}$. By Theorem~\ref{thm:cf} and the above discussion, there is a (unital) classifiable $\cs$-algebraic quantum metric Bauer simplex $A_{S^1}$ and a family of tracial quantum systems $(A_{S^1},\alpha_t)_{t \in [0,1]}$ associated with $(S^1,g_t \circ h_\theta \circ g_t^{-1})_{t \in [0,1]}$ with the properties stated in the theorem.

We call the tracial quantum crossed product $A_{S^1} \rtimes_{\alpha_{h_\theta}} \nn$ associated with $(S^1,h_\theta)$ a \emph{tracial rotation algebra}. Note that, if $A_{S^1}=\js_{S^1}$ is the projectionless model described in Example~\ref{ex:cqmcs}.\ref{it:ex2}, then as mentioned in \cite[Example 4.10]{Jacelon:2022wr} one can compute that
\begin{equation} \label{eqn:ksphere}
\left(K_0(A_{S^1}),K_0(A_{S^1})_+,[1]\right) \cong (\zz,\nn,1).
\end{equation}
Note also that, regardless of whether $\theta\in(0,1)$ is rational or irrational, $A_{S^1} \rtimes_{\alpha_{h_\theta}} \nn$ is classifiable (a consequence of finite Rokhlin dimension, which is a regularity property of the tracial quantum system that we have prioritised over classical dynamical features like topological minimality). It is also a quantum metric \emph{Bauer} simplex, the boundary of the trace space being either a singleton (if $\theta$ is irrational) or homeomorphic to the circle via the set of finitely supported periodic measures (if $\theta$ is rational). The other fibres $A_{S^1} \rtimes_{\alpha_t} \nn$ are deformed tracial rotation algebras whose isometric isomorphism classes vary continuously relative to the quantum intertwining gap $\gamma_q$ (and also the quantum Gromov--Hausdorff distance). Indeed,
\[
(g_s\circ g_t^{-1})_* \colon T(A_{S^1} \rtimes_{\alpha_t} \nn) = (g_t)_*\left(\prob(X)^{h_\theta}\right) \to (g_s)_*\left(\prob(X)^{h_\theta}\right) = T(A_{S^1} \rtimes_{\alpha_s} \nn)
\]
is invertible, and is also approximately isometric (in the sense of Definition~\ref{def:quig}) if $s$ is close to $t$.

 We can also arrange for the fibres to be $K$-connected. Suppose, for example, that instead of $K_1(\alpha_t)=\id_{K_1(A)}$ as in the proof of Theorem~\ref{thm:cf} we demand that $K_1(\alpha_t)=0$ (our  replacement for a technique like topological orbit breaking that would be used to manage $K_1$ in the commutative setting). We would then lose \eqref{it:field0}, but again as mentioned in \cite[Example 4.10]{Jacelon:2022wr} we would have from the Pimsner--Voiculescu sequence \eqref{eqn:pv} that
\[
\left(K_0(A_{S^1} \rtimes_{\alpha_t} \nn),K_0(A_{S^1} \rtimes_{\alpha_t} \nn)_+,[1],K_1(A_{S^1} \rtimes_{\alpha_t} \nn)\right) \cong (\zz,\nn,1,\zz). 
\]
In particular, as in Corollary~\ref{cor:quig}, $\gamma_q$ is a quasimetric on the isometric isomorphism classes of fibres.

Note that the first half of Theorem~\ref{thm:cf} also applies to the more na\"{i}ve deformation $h_\theta \mapsto h_{t\theta}$, for fixed $\theta$ and $t\in[0,1]$. We would again obtain a continuous field of tracial rotation algebras. But we do not have equivariant Gromov--Hausdorff continuity, and we do not have a continuous $\cs$-QMCS field. That said, it is not hard to show for this field that irrational values of $\theta$ are continuity points of the quantum distance functions \eqref{eqn:qdf}.
\end{example}

\begin{example} \label{ex:sphere}
By \cite{Asimov:1976aa}, any finite group $G$ of $k+1$ elements can be exhibited as the isometry group of a Riemannian sphere $S^{k-1}$. Conjugating $G$ by a path $g_t$ of diffeomorphisms as in Example~\ref{ex:rotation} provides a uniformly Lipschitz family of $G$-actions whose simplices of invariant measures vary Hausdorff continuously. Moreover, as discussed in Remark~\ref{rem:groups}, we can (fibrewise) lift each $G$-action to a $G$-action $\alpha_t$ on the classifiable $\cs$-algebra $\js_{S^{k-1}}\otimes\mathcal{Q}$ witnessing the tracial dynamics. Whether or not we can find a family of lifts that varies continuously with $t$ is a question better left for the authors of \cite{Carrion:2024aa,Carrion:wz}, but with the current technology we can at least touch on the case $G=\zz/(k+1)\zz$, that is, the case of periodic $\zz$-actions. Here, we get a continuous field of (unital) classifiable $\cs$-algebraic quantum metric Choquet simplices $\js_{S^{k-1}} \rtimes_{\alpha_t} \nn$, where at every fibre $t$, $\alpha_t$ is approximately unitarily equivalent to an automorphism whose $(k+1)$-th power is approximately unitarily equivalent to the identity.
\end{example}


\begin{thebibliography}{10}

\bibitem{Alfsen:1971hl}
E.~M. Alfsen.
\newblock {\em Compact convex sets and boundary integrals}.
\newblock Springer-Verlag, New York, 1971.
\newblock Ergebnisse der Mathematik und ihrer Grenzgebiete, Band 57.

\bibitem{Delaroche:1987aa}
C.~Anantharaman-Delaroche.
\newblock Syst\`emes dynamiques non commutatifs et moyennabilit\'e.
\newblock {\em Math. Ann.}, 279(2):297--315, 1987.

\bibitem{Asimov:1976aa}
D.~Asimov.
\newblock Finite groups as isometry groups.
\newblock {\em Trans. Amer. Math. Soc.}, 216:388--390, 1976.

\bibitem{Barlak:2017aa}
S.~Barlak and G.~Szab\'{o}.
\newblock Rokhlin actions of finite groups on {UHF}-absorbing {$\cs$}-algebras.
\newblock {\em Trans. Amer. Math. Soc.}, 369(2):833--859, 2017.

\bibitem{Benaim:0aa}
M.~Bena\"{\i}m and T.~Hurth.
\newblock {\em Markov chains on metric spaces---a short course}.
\newblock Universitext. Springer, Cham, 2022.

\bibitem{Benoist:2016ur}
Y.~Benoist and J.-F. Quint.
\newblock Central limit theorem for linear groups.
\newblock {\em Ann. Probab.}, 44(2):1308--1340, 2016.

\bibitem{Benoist:2016ab}
Y.~Benoist and J.-F. Quint.
\newblock {\em Random walks on reductive groups}, volume~62 of {\em Ergebnisse
  der Mathematik und ihrer Grenzgebiete. 3. Folge. A Series of Modern Surveys
  in Mathematics [Results in Mathematics and Related Areas. 3rd Series. A
  Series of Modern Surveys in Mathematics]}.
\newblock Springer, Cham, 2016.

\bibitem{Blackadar:1998qf}
B.~Blackadar.
\newblock {\em {$K$}-theory for operator algebras}, volume~5 of {\em
  Mathematical Sciences Research Institute Publications}.
\newblock Cambridge University Press, Cambridge, second edition, 1998.

\bibitem{Blackadar:1980zr}
B.~E. Blackadar.
\newblock Traces on simple {AF} {$\cs$}-algebras.
\newblock {\em J. Funct. Anal.}, 38(2):156--168, 1980.

\bibitem{Bosa:aa}
J.~Bosa, N.~P. Brown, Y.~Sato, A.~Tikuisis, S.~White, and W.~Winter.
\newblock Covering dimension of {$\cs$}-algebras and 2-coloured
  classification.
\newblock {\em Mem. Amer. Math. Soc.}, 257(1233):vii+97, 2019.

\bibitem{Breiman:1960aa}
L.~Breiman.
\newblock The strong law of large numbers for a class of {M}arkov chains.
\newblock {\em Ann. Math. Statist.}, 31:801--803, 1960.

\bibitem{Carrion:2024aa}
J.~R. Carri\'{o}n, J.~Gabe, C.~Schafhauser, A.~Tikuisis, and S.~White.
\newblock Classification of $^*$-homomorphisms {{\rm II}}.
\newblock In preparation.

\bibitem{Carrion:wz}
J.~R. Carri\'{o}n, J.~Gabe, C.~Schafhauser, A.~Tikuisis, and S.~White.
\newblock Classification of $^*$-homomorphisms {{\rm I}}: {S}imple nuclear
  {$\cs$}-algebras.
\newblock arXiv:2307.06480 [math.OA], 2023.

\bibitem{Chung:2020aa}
N.-P. Chung.
\newblock Gromov-{H}ausdorff distances for dynamical systems.
\newblock {\em Discrete Contin. Dyn. Syst.}, 40(11):6179--6200, 2020.

\bibitem{Cortez:2008aa}
M.~I. Cortez and S.~Petite.
\newblock {$G$}-odometers and their almost one-to-one extensions.
\newblock {\em J. Lond. Math. Soc. (2)}, 78(1):1--20, 2008.

\bibitem{Cuntz:1982uk}
J.~Cuntz.
\newblock The internal structure of simple {$\cs$}-algebras.
\newblock In {\em Operator algebras and applications, {P}art {I} ({K}ingston,
  {O}nt., 1980)}, volume~38 of {\em Proc. Sympos. Pure Math.}, pages 85--115.
  Amer. Math. Soc., Providence, R.I., 1982.

\bibitem{Cuntz:1979fv}
J.~Cuntz and G.~K. Pedersen.
\newblock Equivalence and traces on {$\cs$}-algebras.
\newblock {\em J. Funct. Anal.}, 33(2):135--164, 1979.

\bibitem{Davidson:2022aa}
K.~R. Davidson and M.~Kennedy.
\newblock Noncommutative {C}hoquet theory.
\newblock arXiv:1905.08436 [math.OA], 2022.

\bibitem{Diaconis:1999aa}
P.~Diaconis and D.~Freedman.
\newblock Iterated random functions.
\newblock {\em SIAM Rev.}, 41(1):45--76, 1999.

\bibitem{Dixmier:1964rt}
J.~Dixmier.
\newblock {\em Les {$\cs$}-alg{\`e}bres et leurs repr{\'e}sentations}.
\newblock Cahiers Scientifiques, Fasc. XXIX. Gauthier-Villars \& Cie,
  {\'E}diteur-Imprimeur, Paris, 1964.

\bibitem{Downarowicz:1991te}
T.~Downarowicz.
\newblock The {C}hoquet simplex of invariant measures for minimal flows.
\newblock {\em Israel J. Math.}, 74(2-3):241--256, 1991.

\bibitem{Dykema:2004aa}
K.~J. Dykema.
\newblock Exactness of reduced amalgamated free product {$\cs$}-algebras.
\newblock {\em Forum Math.}, 16(2):161--180, 2004.

\bibitem{Elliott:2016ab}
G.~A. Elliott, G.~Gong, H.~Lin, and Z.~Niu.
\newblock On the classification of simple amenable {$\cs$}-algebras with finite
  decomposition rank, {\rm{II}}.
\newblock arXiv:1507.03437 [math.OA], 2016.

\bibitem{Elliott:2020wc}
G.~A. Elliott, G.~Gong, H.~Lin, and Z.~Niu.
\newblock The classification of simple separable {KK}-contractible {$\cs$}-algebras with finite nuclear dimension.
\newblock {\em J. Geom. Phys.}, 158:103861, 51, 2020.

\bibitem{Elliott:2020vp}
G.~A. Elliott, G.~Gong, H.~Lin, and Z.~Niu.
\newblock Simple stably projectionless {$\cs$}-algebras with generalized
  tracial rank one.
\newblock {\em J. Noncommut. Geom.}, 14(1):251--347, 2020.

\bibitem{Elliott:2023aa}
G.~A. Elliott, C.~G. Li, and Z.~Niu.
\newblock Remarks on {V}illadsen algebras.
\newblock {\em J. Funct. Anal.}, 287(7):Paper No. 110547, 2024.

\bibitem{Fukaya:1990aa}
K.~Fukaya.
\newblock Hausdorff convergence of {R}iemannian manifolds and its applications.
\newblock In {\em Recent topics in differential and analytic geometry}, volume
  18-{\rm I} of {\em Adv. Stud. Pure Math.}, pages 143--238. Academic Press,
  Boston, MA, 1990.

\bibitem{Gardella:2023aa}
E.~Gardella, S.~Geffen, J.~Kranz, and P.~Naryshkin.
\newblock Classifiability of crossed products by nonamenable groups.
\newblock {\em J. Reine Angew. Math.}, 797:285--312, 2023.

\bibitem{Gardella:2022ab}
E.~Gardella, S.~Geffen, J.~Kranz, P.~Naryshkin, and A.~Vaccaro.
\newblock Tracially amenable actions and purely infinite crossed products.
\newblock arXiv:2211.16872 [math.OA], to appear in Math. Ann., 2022.

\bibitem{Gardella:2021tb}
E.~Gardella, I.~Hirshberg, and A.~Vaccaro.
\newblock Strongly outer actions of amenable groups on {$\mathcal{Z}$}-stable
  nuclear {$\cs$}-algebras.
\newblock {\em J. Math. Pures Appl. (9)}, 162:76--123, 2022.

\bibitem{Gong:2020vg}
G.~Gong and H.~Lin.
\newblock On classification of non-unital amenable simple {$\cs$}-algebras,
  {II}.
\newblock {\em J. Geom. Phys.}, 158:103865, 102, 2020.

\bibitem{Gong:2021tw}
G.~Gong and H.~Lin.
\newblock On classification of non-unital amenable simple {$\cs$}-algebras,
  {III}, {S}tably projectionless {$\cs$}-algebras.
\newblock arXiv:2112.14003 [math.OA], 2021.

\bibitem{Gong:2020ud}
G.~Gong, H.~Lin, and Z.~Niu.
\newblock A classification of finite simple amenable {$\mathcal Z$}-stable
  {$\cs$}-algebras, {I}: {$\cs$}-algebras with generalized tracial rank
  one.
\newblock {\em C. R. Math. Acad. Sci. Soc. R. Can.}, 42(3):63--450, 2020.

\bibitem{Gong:2020uf}
G.~Gong, H.~Lin, and Z.~Niu.
\newblock A classification of finite simple amenable {$\mathcal Z$}-stable
  {$\cs$}-algebras, {II}: {$\cs$}-algebras with rational
  generalized tracial rank one.
\newblock {\em C. R. Math. Acad. Sci. Soc. R. Can.}, 42(4):451--539, 2020.

\bibitem{Gromov:1981aa}
M.~Gromov.
\newblock Groups of polynomial growth and expanding maps.
\newblock {\em Inst. Hautes \'{E}tudes Sci. Publ. Math.}, (53):53--73, 1981.

\bibitem{Haagerup:1978aa}
U.~Haagerup.
\newblock An example of a nonnuclear {$\cs$}-algebra, which has the
  metric approximation property.
\newblock {\em Invent. Math.}, 50(3):279--293, 1978/79.

\bibitem{Heinonen:2001aa}
J.~Heinonen.
\newblock {\em Lectures on analysis on metric spaces}.
\newblock Universitext. Springer-Verlag, New York, 2001.

\bibitem{Higson:2001aa}
N.~Higson and G.~Kasparov.
\newblock {$E$}-theory and {$KK$}-theory for groups which act properly and
  isometrically on {H}ilbert space.
\newblock {\em Invent. Math.}, 144(1):23--74, 2001.

\bibitem{Hirshberg:2015wh}
I.~Hirshberg, W.~Winter, and J.~Zacharias.
\newblock Rokhlin dimension and {$\cs$}-dynamics.
\newblock {\em Comm. Math. Phys.}, 335(2):637--670, 2015.

\bibitem{Jacelon:2022wr}
B.~Jacelon.
\newblock Chaotic tracial dynamics.
\newblock {\em Forum Math. Sigma}, 11:Paper No. e39, 21, 2023.

\bibitem{Jacelon:2021vc}
B.~Jacelon.
\newblock Metrics on trace spaces.
\newblock {\em J. Funct. Anal.}, 285(4):Paper No. 109977, 32, 2023.

\bibitem{Jacelon:2021wa}
B.~Jacelon, K.~R. Strung, and A.~Vignati.
\newblock Optimal transport and unitary orbits in {$\cs$}-algebras.
\newblock {\em J. Funct. Anal.}, 281(5):109068, 2021.

\bibitem{Jacelon:2021uc}
B.~Jacelon and A.~Vignati.
\newblock Stably projectionless {F}ra\"{\i}ss\'{e} limits.
\newblock {\em Studia Math.}, 267(2):161--199, 2022.

\bibitem{Jiang:1999hb}
X.~Jiang and H.~Su.
\newblock On a simple unital projectionless {$\cs$}-algebra.
\newblock {\em Amer. J. Math.}, 121(2):359--413, 1999.

\bibitem{Kaad:2021aa}
J.~Kaad and D.~Kyed.
\newblock Dynamics of compact quantum metric spaces.
\newblock {\em Ergodic Theory Dynam. Systems}, 41(7):2069--2109, 2021.

\bibitem{Kennedy:2022aa}
M.~Kennedy and E.~Shamovich.
\newblock Noncommutative {C}hoquet simplices.
\newblock {\em Math. Ann.}, 382(3-4):1591--1629, 2022.

\bibitem{Kerr:2009aa}
D.~Kerr and H.~Li.
\newblock On {G}romov--{H}ausdorff convergence for operator metric spaces.
\newblock {\em J. Operator Theory}, 62(1):83--109, 2009.

\bibitem{Kishimoto:1981tg}
A.~Kishimoto.
\newblock Outer automorphisms and reduced crossed products of simple {$\cs$}-algebras.
\newblock {\em Comm. Math. Phys.}, 81(3):429--435, 1981.

\bibitem{Latremoliere:2016wn}
F.~Latr\'{e}moli{\`e}re.
\newblock Quantum metric spaces and the {G}romov--{H}ausdorff propinquity.
\newblock In {\em Noncommutative geometry and optimal transport}, volume 676 of
  {\em Contemp. Math.}, pages 47--133. Amer. Math. Soc., Providence, RI, 2016.

\bibitem{Li:2003aa}
H.~Li.
\newblock {$\cs$}-algebraic quantum {G}romov--{H}ausdorff distance.
\newblock arXiv:math/0312003 [math.OA], 2003.

\bibitem{Li:2006aa}
H.~Li.
\newblock Order-unit quantum {G}romov--{H}ausdorff distance.
\newblock {\em J. Funct. Anal.}, 231(2):312--360, 2006.

\bibitem{Li:2009aa}
H.~Li.
\newblock Compact quantum metric spaces and ergodic actions of compact quantum
  groups.
\newblock {\em J. Funct. Anal.}, 256(10):3368--3408, 2009.

\bibitem{Liao:2016ui}
H.-C. Liao.
\newblock A {R}okhlin type theorem for simple {$\cs$}-algebras of finite
  nuclear dimension.
\newblock {\em J. Funct. Anal.}, 270(10):3675--3708, 2016.

\bibitem{Lin:2007qf}
H.~Lin.
\newblock Simple nuclear {$\cs$}-algebras of tracial topological rank one.
\newblock {\em J. Funct. Anal.}, 251(2):601--679, 2007.

\bibitem{Meyer:2006aa}
R.~Meyer and R.~Nest.
\newblock The {B}aum--{C}onnes conjecture via localisation of categories.
\newblock {\em Topology}, 45(2):209--259, 2006.

\bibitem{Michael:1956aa}
E.~Michael.
\newblock Continuous selections. {I}.
\newblock {\em Ann. of Math. (2)}, 63:361--382, 1956.

\bibitem{Ozawa:2021aa}
N.~Ozawa and Y.~Suzuki.
\newblock On characterizations of amenable {$\cs$}-dynamical systems and new examples.
\newblock {\em Sel. Math., New Ser.}, 27(5):92, 2021.

\bibitem{Pedersen:1969kq}
G.~K. Pedersen.
\newblock Measure theory for {$\cs$}-algebras. {III}.
\newblock {\em Math. Scand.}, 25:71--93, 1969.

\bibitem{Phelps:2001rz}
R.~R. Phelps.
\newblock {\em Lectures on {C}hoquet's theorem}, volume 1757 of {\em Lecture
  Notes in Mathematics}.
\newblock Springer-Verlag, Berlin, second edition, 2001.

\bibitem{Pimsner:1980yu}
M.~Pimsner and D.~Voiculescu.
\newblock Exact sequences for {$K$}-groups and {E}xt-groups of certain
  cross-product {$\cs$}-algebras.
\newblock {\em J. Operator Theory}, 4(1):93--118, 1980.

\bibitem{Pimsner:1982aa}
M.~Pimsner and D.~Voiculescu.
\newblock {$K$}-groups of reduced crossed products by free groups.
\newblock {\em J. Operator Theory}, 8(1):131--156, 1982.

\bibitem{Rieffel:1989ly}
M.~A. Rieffel.
\newblock Continuous fields of {$\cs$}-algebras coming from group cocycles and
  actions.
\newblock {\em Math. Ann.}, 283(4):631--643, 1989.

\bibitem{Rieffel:1998ww}
M.~A. Rieffel.
\newblock Metrics on states from actions of compact groups.
\newblock {\em Doc. Math.}, 3:215--229, 1998.

\bibitem{Rieffel:1999aa}
M.~A. Rieffel.
\newblock Metrics on state spaces.
\newblock {\em Doc. Math.}, 4:559--600, 1999.

\bibitem{Rieffel:2002aa}
M.~A. Rieffel.
\newblock Group {$\cs$}-algebras as compact quantum metric spaces.
\newblock {\em Doc. Math.}, 7:605--651, 2002.

\bibitem{Rieffel:2004aa}
M.~A. Rieffel.
\newblock Gromov--{H}ausdorff distance for quantum metric spaces.
\newblock {\em Mem. Amer. Math. Soc.}, 168(796):1--65, 2004.
\newblock Appendix 1 by Hanfeng Li.

\bibitem{Rieffel:2010ab}
M.~A. Rieffel.
\newblock Leibniz seminorms for ``matrix algebras converge to the sphere''.
\newblock In {\em Quanta of maths}, volume~11 of {\em Clay Math. Proc.}, pages
  543--578. Amer. Math. Soc., Providence, RI, 2010.

\bibitem{Robert:2009aa}
L.~Robert.
\newblock On the comparison of positive elements of a {$\cs$}-algebra by lower
  semicontinuous traces.
\newblock {\em Indiana Univ. Math. J.}, 58(6):2509--2515, 2009.

\bibitem{Rong:2010aa}
X.~Rong.
\newblock Convergence and collapsing theorems in {R}iemannian geometry.
\newblock In {\em Handbook of geometric analysis, {N}o. 2}, volume~13 of {\em
  Adv. Lect. Math. (ALM)}, pages 193--299. Int. Press, Somerville, MA, 2010.

\bibitem{Rordam:1995fr}
M.~R{\o}rdam.
\newblock Classification of certain infinite simple {$\cs$}-algebras.
\newblock {\em J. Funct. Anal.}, 131(2):415--458, 1995.

\bibitem{Rordam:2002yu}
M.~R{\o}rdam and E.~St{\o}rmer.
\newblock {\em Classification of nuclear {$\cs$}-algebras. {E}ntropy in
  operator algebras}, volume 126 of {\em Encyclopaedia of Mathematical
  Sciences}.
\newblock Springer-Verlag, Berlin, 2002.
\newblock Operator Algebras and Non-commutative Geometry, 7.

\bibitem{Stacey:1993uc}
P.~J. Stacey.
\newblock Crossed products of {$\cs$}-algebras by {$^*$}-endomorphisms.
\newblock {\em J. Austral. Math. Soc. Ser. A}, 54(2):204--212, 1993.

\bibitem{Szabo:2019te}
G.~Szab\'{o}, J.~Wu, and J.~Zacharias.
\newblock Rokhlin dimension for actions of residually finite groups.
\newblock {\em Ergodic Theory Dynam. Systems}, 39(8):2248--2304, 2019.

\bibitem{Thoma:1964aa}
E.~Thoma.
\newblock \"{U}ber unit\"{a}re {D}arstellungen abz\"{a}hlbarer, diskreter
  {G}ruppen.
\newblock {\em Math. Ann.}, 153:111--138, 1964.

\bibitem{Thomsen:1994qy}
K.~Thomsen.
\newblock Inductive limits of interval algebras: the tracial state space.
\newblock {\em Amer. J. Math.}, 116(3):605--620, 1994.

\bibitem{Ursu:2021wp}
D.~Ursu.
\newblock Characterizing traces on crossed products of noncommutative
  {$\cs$}-algebras.
\newblock {\em Adv. Math.}, 391:Paper No. 107955, 29, 2021.

\bibitem{Xia:2009aa}
Q.~Xia.
\newblock The geodesic problem in quasimetric spaces.
\newblock {\em J. Geom. Anal.}, 19(2):452--479, 2009.

\bibitem{Yannelis:1983aa}
N.~C. Yannelis and N.~D. Prabhakar.
\newblock Existence of maximal elements and equilibria in linear topological
  spaces.
\newblock {\em J. Math. Econom.}, 12(3):233--245, 1983.

\end{thebibliography}
\end{document}